\documentclass[]{article}

\everypar{\looseness=-1}
\usepackage{iclr2026_conference}
\usepackage[dvipsnames]{xcolor}
\usepackage[utf8]{inputenc} 
\usepackage[T1]{fontenc}    
\usepackage{natbib}
\usepackage{hyperref}       
\usepackage{url}            
\usepackage{pbox}
\usepackage{booktabs}       
\usepackage{amsfonts}       
\usepackage{nicefrac}       
\usepackage{microtype}      
\usepackage{xcolor}         
\usepackage{comment}
\newtheorem{assumption}{Assumption}

\usepackage{url}            
\usepackage{booktabs}       
\usepackage{amsfonts}       
\usepackage{nicefrac}       
\usepackage{microtype}      

\usepackage{amssymb, amsmath, latexsym}
\usepackage{url}

\usepackage{algorithm}
\usepackage{algorithmic}

\usepackage{tabularx}
\usepackage{paralist}
\usepackage{mathtools}

\usepackage{bbm} 
\usepackage{wrapfig}
\usepackage{makecell}
\usepackage{multirow}
\usepackage{booktabs}
\usepackage{nicefrac}       

\usepackage[flushleft]{threeparttable} 

\usepackage{caption}
\usepackage{multirow}
\usepackage{colortbl}
\definecolor{bgcolor}{rgb}{0.8,1,1}
\definecolor{bgcolor2}{rgb}{0.8,1,0.8}
\definecolor{niceblue}{rgb}{0.0,0.19,0.56}

\usepackage{hyperref}
\hypersetup{colorlinks,linkcolor={blue},citecolor={niceblue},urlcolor={blue}}

\renewcommand{\algorithmicrequire}{\textbf{Input:}}
\renewcommand{\algorithmicensure}{\textbf{Output:}}

\usepackage{pifont}
\definecolor{PineGreen}{RGB}{0,110,51}
\definecolor{BrickRed}{RGB}{143,20,2}

\usepackage{tikz-cd} 

\usepackage{subfigure}

\newcommand{\R}{\mathbb{R}}
\newcommand{\eqdef}{\stackrel{\text{def}}{=}}

\def\<#1,#2>{\left\langle #1,#2\right\rangle}

\newcolumntype{Y}{>{\centering\arraybackslash}X}

\usepackage{xspace}

\newcommand{\algname}[1]{{\sf  #1}\xspace}

\usepackage[colorinlistoftodos,bordercolor=orange,backgroundcolor=orange!20,linecolor=orange,textsize=scriptsize]{todonotes}

\newcommand{\cO}{{\cal O}}

\newcommand{\EE}{\mathbb{E}}

\def\clip{\texttt{clip}}
\def\med{\texttt{Med}}
\def\batch{\texttt{BatchMed}}

\usepackage{hyperref}
\graphicspath{{plots/}}

\usepackage{makecell}

\usepackage{accents}
\newlength{\dhatheight}

\usepackage{pgfplotstable} 
\usetikzlibrary{automata, positioning, arrows, shapes, fit, calc, intersections}
\usepgfplotslibrary{statistics}

\def\la{\langle}
\def\ra{\rangle}
\iclrfinalcopy
\usepackage{enumitem}
\usepackage{wrapfig}
\usepackage{natbib}

\newtheorem{theorem}{Theorem}

\newtheorem{lemma}{Lemma}
\newtheorem{proposition}{Proposition}

\newtheorem{remark}{Remark}
\newenvironment{proof}{%
  \par\noindent\textbf{Proof.}\quad%
}{%
  \hfill$\square$%
}

\usepackage{xcolor}

\begin{document}
\author{Nikita Kornilov \\MIPT, Skoltech
\And
Yuriy Dorn  \\
MSU AI Institute, MIPT, IITP RAS
\And
Aleksandr Lobanov \\
MIPT, ISP RAS
\And
Nikolay Kutuzov \\
RCenter for AI, MIPT 
\And
Innokentiy Shibaev \\
MIPT, IITP RAS
\And
Eduard Gorbunov \\
MBZUAI
\And
Alexander Nazin \\
ICS RAS, MIPT
\And
Alexander Gasnikov \\
Innopolis University, Steklov MI RAS, MIPT
}

\title{Median Clipping for Zeroth-order Non-Smooth Optimization and Multi-Armed Bandit}

\maketitle
\begin{abstract}
In this paper, we consider non-smooth convex optimization with a zeroth-order oracle corrupted by symmetric stochastic noise. Unlike the existing high-probability results requiring the noise to have bounded $\kappa$-th moment with $\kappa \in (1,2]$, our results allow even heavier noise with any $\kappa > 0$, e.g., the noise distribution can have unbounded expectation. Our convergence rates match the best-known ones for the case of the bounded variance, namely, to achieve function accuracy $\varepsilon$ our methods with Lipschitz oracle require  $\tilde{O}(d^2\varepsilon^{-2})$ iterations for any $\kappa > 0$. We build the median gradient estimate with bounded second moment as the mini-batched median of the sampled gradient differences. We apply this technique to the stochastic multi-armed bandit problem with heavy-tailed distribution of rewards and achieve $\tilde{O}(\sqrt{dT})$ regret. We demonstrate the performance of our zeroth-order and MAB algorithms for various $\kappa \in (0,2]$ on synthetic and real-world data. Our methods do not lose to SOTA approaches and dramatically outperform them for $\kappa \leq 1$.
\end{abstract}

\vspace{-0.2cm}
\section{Introduction}\label{sec:intro}
\vspace{-0.15cm}
During the recent few years, stochastic optimization problems with heavy-tailed noise received a lot of attention from many researchers. In particular, heavy-tailed noise is observed in various problems, such as the training of large language models \citep{brown2020language, zhang2020adaptive}, generative adversarial networks \citep{goodfellow2014generative, gorbunov2022clipped}, finance \citep{rachev2003handbook}, and blockchain \citep{wang2019flash}. The noise is called heavy-tailed if it has bounded $\kappa$-th moment for $\kappa \in (1,2]$. In particular case of $\kappa = 2$, it has bounded variance and considered to be light-tailed. 

One of the most popular techniques for handling heavy-tailed noise in theory and practice is  the gradient clipping \citep{gorbunov2020stochastic, cutkosky2021high, nguyen2023improved, puchkin2023breaking} which allows deriving high-probability bounds and considerably improves convergence even in case of light tails \citep{sadiev2023high}. For convex functions,  \citep{sadiev2023high, nguyen2023improved} show that first-order methods with clipping and properly adjusted clipping levels, stepsizes achieves the optimal sample complexity bound $\tilde{O}(\varepsilon^{-\frac{\kappa}{\kappa - 1}})$. In \citep{sadiev2023high}, the authors propose to use restarts to accelerate the methods in case of strongly convex functions, namely, they obtain optimal  $\tilde{O}(\varepsilon^{-\frac{\kappa}{2(\kappa - 1)}})$ rates.


However, most of the mentioned works focus on the gradient-based (first-order) methods. For some problems, e.g., the multi-armed bandit \citep{flaxman2004online,bartlett2008high, liu2011multi, bubeck2013bandits}, only losses or function values are available, and thus, zeroth-order algorithms are required. Stochastic zeroth-order optimization is being actively studied. For a detailed overview, see the recent survey \citep{gasnikov2022randomized} and the references
therein. The only existing works that handle heavy-tailed noise in convex zeroth-order optimization are \citep{kornilov2023gradient, kornilov2023accelerated} which combine clipping and gradient smoothing \citep{gasnikov2022power} techniques. The authors obtain optimal high-probability convergence for $d$-dimensional non-smooth convex  problems, i.e., function accuracy $\varepsilon$ is achieved in $\tilde{O} ((\sqrt{d} \varepsilon^{-1})^\frac{\kappa}{\kappa - 1}) \text{oracle calls}.$ These rates match the optimal rates for first-order optimization  \citep{gorbunov2020stochastic} in $\varepsilon$, however, they degenerate as $\kappa \to 1$, and the convergence is not guaranteed for $\kappa = 1$. The same is related to the degenerating $\tilde{O} ((d\varepsilon^{-1})^\frac{\kappa}{2(\kappa - 1)})$ rates for strongly convex functions.

In optimization literature \citep{jakovetic2023nonlinear, armacki2023high, armacki2024large, compagnoni2024adaptive, compagnoni2025unbiased}, it was observed that for particular class of \textit{symmetric} heavy-tailed noise the first-order methods do not suffer from small $\kappa \to 1$ and can even work under noises without finite math expectation.  For example, under symmetric (and close to symmetric) heavy-tailed noises, the degeneration issue can be handled via median estimates \citep{zhong2021breaking, puchkin2023breaking}, which are frequently used in robust mean estimation and robust machine learning \citep{lugosi2019mean}. In the case of first-order methods, the authors of \citep{puchkin2023breaking} combine clipping with median estimate and achieve better complexity guarantees $\tilde{O}(\varepsilon^{-2}/\kappa)$ and $\tilde{O}(\varepsilon^{-1}/\kappa)$ for convex and strongly convex functions, respectively. They also show that the narrowing of the distributions' class is essential for breaking the lower bounds. However, the possibility of application of the median estimates to the case of the zeroth-order optimization and multi-armed bandit remains open. In this paper, we address this question.


\subsection{Contributions}
\textbf{Theory I: novel oracle concept.} We propose our novel theoretical zeroth-order oracle (As.~\ref{as:oracle}) that allows us to incorporate fine-grained features of the noise probability distributions. We use it to successfully utilize symmetry of the heavy-tailed noise and dramatically improve current convergence results. \color{black}

\noindent \textbf{Theory II: zeroth-order optimization.} We propose our novel \algname{ZO-clipped-med-SSTM} ($\S$\ref{sec:SSTM}) for unconstrained optimization and \algname{ZO-clipped-med-SMD} ($\S$\ref{sec:SMD}) for optimization on convex compact which successfully incorporate median clipping technique. For any symmetric heavy-tailed noise with bounded $\kappa$-th moment $\kappa > 0$, our methods achieve not degenerating convergence rates with high-probability which match the optimal rates for ZO minimization under any noise with the bounded variance. For $\mu$-strongly convex functions, we use restart technique to accelerate our algorithms (Appendix \ref{par:Restarts}). In the Table~\ref{tab:res_uncon}, we provide convergence guarantees for the unconstrained case. 

\begin{table*}[t]
\centering
\footnotesize
\caption{Number of successive iterations to achieve a function's accuracy $\varepsilon$ with high probability; unconstrained optimization via Lipschitz oracle with bounded $\kappa$-th moment.  Constants $b, d, M'_2$ denote the batch size, dimensionality, and the Lipschitz constant of the oracle, respectively.} 

\label{tab:res_uncon}
\begin{tabular}{|l|ll|}
\hline

                 \multirow{2}*{Setup} & \multicolumn{1}{l|}{\algname{ZO-clipped-SSTM} \citep{kornilov2023accelerated}}      &  \color{niceblue} {Our \algname{ZO-clipped-med-SSTM}}   \\    
                  & \multicolumn{1}{l|}{ $\kappa > 1$,  $b$ oracle calls per iter.}      &  \color{niceblue} $\kappa > 0$, symmetric noise, $\frac{b}{\kappa}$ calls \\
                  \hline
Convex        & \multicolumn{1}{l|}{$\widetilde\cO\left( \max\left\{\frac{d^\frac14 M'_2 }{\varepsilon},\frac1b \left(\frac{\sqrt{d} M'_2 }{\varepsilon}\right)^{\frac{\kappa}{\kappa-1}} \right\}\right)$}      &        \cellcolor{bgcolor} \pbox{20cm}{$\widetilde{\cO}\left( \max\left\{\frac{d^\frac14 M'_2 }{\varepsilon},\frac1b  \left(\frac{dM'_2 }{ \varepsilon} \right)^2 \right\} \right) $ \\ \text{(Theorem \ref{thm:SSTM col})} }      \\ \hline
 \pbox{20cm}{$\mu$-str. \\ convex } &  \multicolumn{1}{l|}{$\widetilde\cO\left(\max\left\{\frac{d^\frac14 M'_2}{\varepsilon}, \frac1b  \left(\frac{d(M'_2)^2}{\mu\varepsilon}\right)^\frac{\kappa}{2(\kappa - 1)}\right\}\right)$ }      &  \cellcolor{bgcolor} \pbox{20cm}{$\widetilde\cO\left( \max\left\{\frac{d^\frac14 M'_2}{\varepsilon},\frac1b  \frac{d^2(M'_2)^2}{\mu\varepsilon}\right\}\right) $ \\ \text{(Theorem \ref{thm:R-med-SSTM cor})}}     \\ \hline
\end{tabular}
\end{table*}

\noindent \textbf{Theory III: Multi Armed Bandit.} We propose \algname{Clipped-INF-med-SMD} ($\S$\ref{sec:MAB}) for the stochastic multi-armed bandit (MAB) with symmetric heavy-tailed reward distribution. For MAB with $d$ arms and time interval $T$, in Theorem \ref{avrregretest}, we obtain the $\tilde{O}(\sqrt{dT})$ bound on the regret, which is optimal and matches the lower bound $\Omega (\sqrt{dT})$ for stochastic MAB with any reward distribution and bounded variance.  Moreover, this bound holds not only in expectation but with controlled large deviations.

\noindent \textbf{Practice.} We demonstrate in experiments ($\S$\ref{sec:Exp}) on extremely noised real and synthetic data superior performance of our methods in comparison with previously known SOTA approaches.  We compare our algorithms with previous approaches and discuss its limitations in   $\S$\ref{sec:limitations}. \color{black}

\section{Preliminaries}
\label{subsec:preliminaries}
In this section, we introduce general notations and assumptions on optimized functions. We also recall popular gradient smoothing and clipping techniques.

\subsection{Notations} For vector $x \in \R^d$ and $p \in [1,2]$, we define $\ell_p$-norm by $\|x\|_p \eqdef \left( \sum\limits_{i=1}^d |x_i|^p\right)^\frac1p$ and its dual norm by $\|x\|_q$, where $\frac1p + \frac1q = 1$. If $q = \infty$, we define $\|x\|_\infty = \max\limits_{i = 1, \dots, d} |x_i|$. We denote the Euclidean ball of radius $R$ and center~$c$: $B_R(c) \eqdef \{ x \in \R^d : \| x - c \|_2 \leq R \}$, the Euclidean sphere: $S_R(c) \eqdef \{ x \in \R^d : \| x - c\|_2 = R \}$ and the probability simplex: $\Delta_+^d \eqdef \{x \in \R^d_+: \sum_{i=1}^d x_i = 1\}$.

Median operator $\text{Median}(\{a_i\}_{i=1}^{2m+1})$ applied to the elements sequence of the odd size $2m~+~1, m \in \mathbb{N}$  returns $m$-th order statistics. We use short notation $a \lor b \eqdef \max(a, b)$.

\subsection{Assumptions}
We consider a non-smooth convex optimization problem on a convex set $Q \subseteq \R^d$:
\begin{equation}\label{eq:min_problem}
    \min\limits_{x\in Q} f(x).
\end{equation}
 A point $x^*$ denotes one of the problem's solutions.
\begin{assumption}[Convexity]\label{as:f convex}
The function $f: Q \to \R$ is $\mu$-strongly convex on $Q \subseteq \R^d$, if there exists a constant $\mu \geq 0$ such that for all $x_1, x_2 \in Q$ and $\lambda \in [0,1]:$
\begin{eqnarray}
    f(\lambda x_1 + (1 - \lambda) x_2 ) &\leq& \lambda f(x_1)
    + (1- \lambda) f(x_2) 
    - \frac{1}{2} \mu \lambda (1 - \lambda) \|x_1  -x_2\|_2^2, \notag
\end{eqnarray}
If $\mu=0$ we say that the function is just “convex”.
\end{assumption}

\begin{assumption}[Lipschitz continuity]\label{as:Lipshcitz}
The function $f: Q \to \R$ is $M_2$-Lipschitz continuous on $Q \subseteq \R^d$, if there exists a constant $M_2 > 0$ such that for all  $x_1, x_2 \in Q$:
\[ |f(x_1) - f(x_2)| \leq M_2 \|x_1 - x_2\|_2.\]
\end{assumption}
If a differentiable function has $L$-Lipschitz gradient, we call it $L$-smooth.

\subsection{Randomized smoothing}
The main scheme that allows us to develop gradient-free
methods for non-smooth convex problems 
is randomized smoothing \cite{ermoliev1976stochastic, nemirovskij1983problem, spall2005introduction, nesterov2017random}.
  For the non-smooth function $f:Q + B_{2\tau}(0) \to \R$, we build the smooth approximation $\hat{f}_\tau: Q \to \R$  with the smoothing parameter $\tau > 0$:
\begin{equation}\label{hat_f}
    \hat{f}_\tau (x) \overset{\text{def}}{=} \EE_{\mathbf{u}\sim  U(B_1(0))} [f(x + \tau \mathbf{u})],
\end{equation}
where $\mathbf{u}$  is a random vector uniformly distributed on the Euclidean unit ball. We define the function $f$ on a slightly larger set $Q + B_{2\tau}(0)$ to be able to compute $\hat{f}_\tau$ on the whole $Q.$
 
If the function $f$ is $\mu$-strongly convex (As. \ref{as:f convex}) and $M_2$-Lipschitz (As. \ref{as:Lipshcitz}), then the smoothed function $\hat{f}_\tau$ is $\mu$-strongly convex and \textit{$\nicefrac{\sqrt{d} M_2}{\tau}$-smooth}. Moreover, it does not differ from the original $f$ too much. These results are formally presented in Lemma \ref{lem:hat_f properties}.
\begin{lemma}[\citep{gasnikov2022power}, Theorem 2.1]\label{lem:hat_f properties}
Consider $\mu$-strongly convex (As. \ref{as:f convex}) and $M_2$-Lipschitz (As. \ref{as:Lipshcitz}) function $f$ on $Q + B_{2\tau}(0) \subseteq \R^d$. For the smoothed function $\hat{f}_\tau$ defined in \eqref{hat_f}, the following properties hold true:
\begin{enumerate}
\item The function $\hat{f}_\tau$ is $M_2$-Lipschitz on $Q$ and satisfies the inequality 
\begin{equation}
    \sup \limits_{x \in Q} |\hat{f}_\tau(x) - f(x)| \leq \tau M_2. \label{eq:sup tauM_2}
\end{equation}
    \item 
The function $\hat{f}_\tau$ is differentiable on $Q$ with the following gradient:
$$\nabla \hat{f}_\tau (x) = \EE_{\mathbf{e}\sim U(S_1(0))}\left[\frac{d}{\tau} f(x + \tau \mathbf{e}) \mathbf{e}\right], \quad x \in Q,$$
where $\mathbf{e}$ is a random vector uniformly distributed on the unit Euclidean sphere. 
\item The function $\hat{f}_\tau$ is $L$-smooth on $Q$ with $L = \nicefrac{\sqrt{d} M_2}{\tau}$.
\end{enumerate}
\end{lemma}

\subsection{Clipping} To handle heavy-tailed noise, we use a clipping technique which clips tails of gradient's distribution. For the clipping level $\lambda >0$ and $\ell_q$-norm, where $q \in [2, +\infty]$, we define the clipping operator $\clip$ for arbitrary non-zero gradient vector $g \in \R^d$ as follows: 
$$\clip_q\left(g, \lambda\right) = \frac{g}{\|g\|_q} \min\left(\|g\|_q, \lambda\right).
$$ 
\section{Zeroth-order optimization with symmetric heavy-tailed noise} \label{sec:Main_Results}

In this section, we present novel algorithms for zeroth-order optimization with independent and Lipschitz oracles. In \S\ref{sec:ZO theory}, we introduce the problem, symmetric heavy-tailed noise assumptions and median estimation with its properties. In \S\ref{sec:SSTM}, we propose our accelerated batched \algname{ZO-clipped-med-SSTM} for unconstrained problems. In \S\ref{sec:SMD}, we describe \algname{ZO-clipped-med-SMD} for problems on convex compacts. All proofs are located in Appendix \ref{sec:proofs}.

\subsection{New zeroth-order noise concept and integration in median estimation}\label{sec:ZO theory}

 \subsubsection{Zeroth-order two-point oracle}
 In zeroth-order setup, the optimization \eqref{eq:min_problem} is performed only by accessing the pairs of function evaluations rather than sub-gradients.

 For any two points $x, y$, an oracle returns the pair of the scalar values $f(x, \xi)$ and $f(y, \xi)$, which are noised evaluation of real values $f(x)$ and $f(y).$ Moreover, noised values have the same realization of the stochastic variable $\xi$ and can be written as  
$$f(x, \xi) - f(y, \xi) = f(x) - f(y)  + \phi(\xi|x,y),$$
where $\phi(\xi| x,y)$ is the stochastic noise, which distribution depends on  $x,y$.

\subsubsection{Symmetric heavy-tailed noise}  We propose our novel assumption on distribution of $\phi(\xi| x,y)$, induced by a random variable $\xi$. It allows us to introduce noise symmetry and heavy tails. 



\begin{assumption}[Symmetric noise distribution]\label{as:dist}
\textbf{I. Symmetry.} For any two points $x,y \in \R^d$,  noise $\phi(\xi| x,y)$ has  symmetric density $p(u|x,y)$, i.e. $p(u|x,y) = p(-u|x,y), \forall u \in \R$. 

\noindent\textbf{II. Heavy tails.} We assume that there exist $\kappa > 0$,  $\gamma > 0$ and scale function $B(x,y): \R^d \times \R^d \to \R$, such that   $\forall u\in \R$ holds
\begin{equation}\label{eq:dist_ass}
    p(u|x,y) \leq \frac{ \gamma^\kappa \cdot | B(x,y)|^{\kappa}}{|B(x,y)|^{1+\kappa} + |u|^{1+\kappa}}.
\end{equation}
We consider two possible oracles:

\textbf{Independent oracle}: $\phi(\xi|x, y)$ distribution doesn't depend on points $x,y$: \begin{equation}\label{eq:Tsubakov}
    \gamma \cdot B(x,y) \equiv \Delta.
\end{equation}
    
\textbf{Lipschitz oracle}: $\phi(\xi|x, y)$ distribution becomes more concentrated around $0$ as points $x$ and $y$ become closer:
\begin{equation}\label{eq:noise_lip}
    |\gamma \cdot B(x,y)| \leq  \Delta \cdot \|x-y\|_2,
\end{equation}
where $\Delta > 0$ is the noise Lipschitz constant.
\end{assumption}
This assumption covers a majority of symmetric absolutely continuous distributions with bounded up to $\kappa$-th moments. If $\xi$ has Cauchy distribution, then one can use  (Remark \ref{rem:examples}):
        \begin{itemize}
            \item Independent oracle: $f(x, \xi) = f(x) + \xi_x, f(y,  \xi) = f(y) + \xi_y$ with independent realizations $\xi_x, \xi_y$.  

        \item Lipschitz oracle: $f(x, \boldsymbol{\xi})\!=\! f(x) + \langle \boldsymbol{\xi}, x \ra , f(y, \boldsymbol{\xi}) \!=\! f(y) + \langle \boldsymbol{\xi}, y \ra$, where $\boldsymbol{\xi}$ is $d$-dimensional random vector. \!\!Oracle gives the same realization of $\boldsymbol{\xi}$ for $x$ and $y$.   
    \end{itemize}

 \noindent \textbf{Comparison with previous oracles.} Our Assumption \ref{as:dist} is quite different from the standard ones from \citep{Dvinskikh_2022, kornilov2023accelerated}.  We make our assumption on variable $\phi(\xi|x, y)$ with fixed $x,y$. It allows us to set and use fine-grained properties of the noise, e.g., symmetry or heavy tails of particular type \eqref{eq:dist_ass}. In \citep{kornilov2023accelerated}, the authors fix $\xi$ and make assumption on $x,y$. Hence, they can not access the distribution of the noise and use only the fact of having bounded $\kappa$-th moment. When $\kappa \in (1,2]$, Assumption \ref{as:dist} can be reduced to the standard one with the same constant, see Remark \ref{rem:diff_as}. 

We would like to highlight the fact that the common proof techniques from previous works can not be trivially generalized to apply symmetry without our novel assumption. For example, the proof of median estimator's properties Lemma \ref{lem:batch median properties} is based on completely different approach. 
We refer to Appendix \ref{par: remarks on as} for more details and the intuition behind Assumption~\ref{as:dist}. \color{black}


\subsubsection{Median estimation} In our pipeline, instead of minimizing the non-smooth function $f$ directly, we propose to minimize the smooth approximation $\hat{f}_\tau$ with the fixed smoothing parameter $\tau$ via first-order methods.  Following \eqref{eq:sup tauM_2}, the solution for $\hat{f}_\tau$  is also approximate minimizer of $f$ when $\tau$ is sufficiently small. Following \citep{shamir2017optimal}, the gradient of $\hat{f}_\tau$ can be estimated by:
\begin{eqnarray}
g(x, \mathbf{e}, \xi) &=& \frac{d}{2\tau}(f(x + \tau \mathbf{e}, \xi) - f(x - \tau \mathbf{e}, \xi)) \mathbf{e}  \label{grad estimation} \\
&=&  \frac{d}{2\tau}(f(x + \tau \mathbf{e}) - f(x - \tau \mathbf{e})  + \phi(\xi|x + \tau \mathbf{e},x - \tau \mathbf{e})) \mathbf{e}, \notag 
\end{eqnarray}
where $\mathbf{e} \sim U(S_1(0))$ is 
a random vector uniformly distributed on the Euclidean unit sphere. Moreover, $\mathbf{e}, \xi$ are independent of each other conditionally on $x$. However, the noise $\phi$ might have unbounded first and second moments. To fix this, we lighten tails of $\phi$ to obtain an unbiased estimate of $\nabla \hat{f}_\tau$. For a point $x$, we apply the component-wise median operator to $2m + 1$ samples  $\{g(x, \mathbf{e}, \xi^i)\}_{i=1}^{2m + 1}$  with independent $\xi^i$ and the same $x $ and $\mathbf{e}$:
\begin{equation} \label{eq:med estimation}
    \med^m(x, \mathbf{e}, \{\xi\}) \eqdef \texttt{Median}(\{g(x, \mathbf{e},  \xi^i)\}_{i=1}^{2m+1}).
\end{equation}
The median operator can be applied to the batch of $\{\mathbf{e}^j\}_{j=1}^b$ with further averaging:
\begin{equation}\label{eq:batch med estimation}
    \batch^{m}_b (x, \{\mathbf{e}\}, \{\xi\}) \eqdef \frac{1}{b} \sum_{j=1}^b \texttt{Med}^m(x, \mathbf{e}^j, \{\xi\}^j).
\end{equation}
 For a large enough number of samples, the median estimates have the bounded second moment. The proof of Lemma~\ref{lem:batch median properties} can be found in Appendix \ref{sec:lemma}.
\begin{lemma}[Median estimate's properties]\label{lem:batch median properties}
Consider $M_2$-Lipschitz (As. \ref{as:Lipshcitz}) function $f$ on $Q + B_{2\tau}(0)$ with oracle corrupted by noise under As. \ref{as:dist} with $\Delta$ and $\kappa > 0$. If median size $m > \frac{2}{\kappa}$ with norm $q \in [2, +\infty]$, then the median estimates \eqref{eq:med estimation} and \eqref{eq:batch med estimation} are unbiased on $Q$:
\begin{equation*}
      \EE_{ \mathbf{e}, \xi}[\batch^{m}_b (x, \{\mathbf{e}\}, \{\xi\})]  = \nabla \hat{f}_\tau(x), 
\end{equation*}
and have bounded second moment, i.e.,
\begin{eqnarray}
    \EE_{\mathbf{e}, \xi}[\|\texttt{BatchMed}^{m}_b (x, \{\mathbf{e}\}, \{\xi\}) -\nabla \hat{f}_\tau(x)\|_2^2] &\leq&  \frac{ \sigma^2}{b},\notag \end{eqnarray}
    \begin{eqnarray}
   \EE_{\mathbf{e}, \xi}[\|\texttt{Med}^{m}(x, \mathbf{e}, \{\xi\}) -\nabla \hat{f}_\tau(x)\|_q^2] \leq   \sigma^2a_q^2, \notag \\
   a_q = d^{\frac1q - \frac12} \min\{\sqrt{32\ln d - 8}, \sqrt{2q - 1}\}. \label{eq:lemma_bounded_mom}
    \end{eqnarray}
For \textbf{independent oracle}, we have $ \sigma^2 = 8d  M_2^2 +  2\left(\frac{d \Delta}{\tau}\right)^2 (2m+1)\left(\frac{4}{\kappa} \right)^{\frac2\kappa},$
and, for \textbf{Lipschitz oracle}, we have  $ \sigma^2 = 8d  M_2^2 +  (16m+8) d^2\Delta^2  \left(\frac{4}{\kappa} \right)^{\frac2\kappa} .$
\end{lemma}

\subsection{Our \algname{ZO-clipped-med-SSTM} for unconstrained problems}\label{sec:SSTM}
We present our novel \algname{ZO-clipped-med-SSTM} that works on the whole space~$\R^d$. We base it on the first-order accelerated clipped Stochastic Similar Triangles Method (\algname{clipped-SSTM}) with the optimal high-probability complexity
bounds from \citep{gorbunov2020stochastic}. Namely, we use its zeroth-order version \algname{ZO-clipped-SSTM} from \citep{kornilov2023accelerated} with the batched median estimation \eqref{eq:batch med estimation}.  The pseudocode is presented in Algorithm \ref{alg:clipped-SSTM}.
\begin{algorithm}[ht!]
\caption{\algname{ZO-clipped-med-SSTM} }
\label{alg:clipped-SSTM}   
\begin{algorithmic}[1]
\REQUIRE Starting point $x^0 \in \R^d$, number of iterations $K$, median size $m$,  batch size $b$, stepsize  $a>0$, smoothing parameter $\tau$, clipping levels $\{\lambda_k\}_{k=0}^{K-1}$.
\STATE Set  $L = \nicefrac{\sqrt{d} M_2}{\tau}$, ~ $A_0 = \alpha_0 = 0$,~ $y^0 = z^0 = x^0.$
\FOR{$k=0,\ldots, K-1$}
\STATE Set $\alpha_{k+1} = \nicefrac{(k+2)}{2aL}, \quad A_{k+1} = A_k + \alpha_{k+1}.$
\STATE $   x^{k+1} = \frac{A_k y^k + \alpha_{k+1} z^k}{A_{k+1}}. $
\STATE Sample sequences $\{\mathbf{e}\} \sim U(S_1(0))$  and $\{\xi\}$ . 
\STATE  \label{eq:SSTM_get_oracle} $g_{med}^{k+1} = \batch^{m}_b (x^{k+1}, \{\mathbf{e}\}, \{\xi\})$. 
\STATE $  z^{k+1} = z^k - \alpha_{k+1} \cdot \clip_2\left( g_{med}^{k+1}, \lambda_{k+1}\right).$
\STATE $   y^{k+1} = \frac{A_k y^k + \alpha_{k+1} z^{k+1}}{A_{k+1}}.$
\ENDFOR
\ENSURE $y^K$ 
\end{algorithmic}
\end{algorithm}


\begin{theorem}[Convergence of \algname{ZO-clipped-med-SSTM}]\label{thm:SSTM col}
    Consider initial point $x_0$ and distance $R = \|x^0 - x^*\|_2$. Let the function $f$ be a convex (As. \ref{as:f convex}) and $M_2$-Lipschitz (As. \ref{as:Lipshcitz}) function on $B_{3R +2\tau}(x^*)$ with two-point oracle corrupted by noise under As. \ref{as:dist} with $\Delta$ and $\kappa > 0$. We set batch size $b$ and median size $m = \frac{2}{\kappa} + 1$. To achieve function accuracy $\varepsilon$, i.e., $f(y^K) - f(x^*) \leq \varepsilon$  with probability at least $1-\beta$   via \algname{ZO-clipped-med-SSTM} with parameters $A = \ln \nicefrac{4K}{\beta}\geq 1$,
$a = \Theta(\min\{A^2, \nicefrac{\sigma K^{2}\sqrt{A}\tau}{\sqrt{bd}M_2R}\}), \lambda_k = \Theta(\nicefrac{R}{(\alpha_{k+1}A)})$ and  smoothing parameter $\tau = \frac{\varepsilon}{4M_2}$,  the number of iterations $K$ must be 
    \begin{eqnarray}\label{eq:SSTM conv iter num lip}
        &&\widetilde{\cO}\left( \frac{d^\frac14 M_2 R}{\varepsilon} \lor  \frac{(\sqrt{d} M_2 R)^2}{b \cdot  \varepsilon^2}  \left( 1  \lor \left(\frac{4}{\kappa} \right)^\frac2\kappa\!\! \frac{d\Delta^2}{\varepsilon^2}\right)\right), \notag \\
        &&\widetilde{\cO}\left( \max\!\left\{\frac{d^\frac14 M_2 R}{\varepsilon},  \frac{d(M_2^2\!\!+\!\!d \Delta^2/\kappa^{\frac{2}{\kappa}})R^2}{b \cdot  \varepsilon^2}  \right\} \right),\notag
    \end{eqnarray}
for \textbf{independent} and \textbf{Lipschitz} oracle, respectively.
 Each iteration requires $(2m + 1) \cdot b $  oracle calls. Moreover, with probability at least  $ 1-\beta$, the iterates of \algname{ZO-clipped-med-SSTM} remain in the ball: $\{x^k\}_{k=0}^{K+1}, \{y^k\}_{k=0}^{K}, \{z^k\}_{k=0}^{K} \subseteq B_{2R}(x^*)$.
\end{theorem}
The proof is located in Appendix \ref{par: sstm conv thms}.
\subsubsection{Discussion}

\textbf{Optimality.} For Lipschitz oracle, the first term matches  the optimal bound in terms of $\varepsilon$ for the deterministic  non-smooth problems \citep{bubeck2019complexity}, and the second term matches the optimal bound for zeroth-order problems with the finite variance \citep{nemirovskij1983problem}.  Under "optimal bound", we mean the optimal bound for the problems with \textit{any noise}. For the symmetric noise only, we are not aware of any bounds.  In terms of $d$, we obtain the factor $dM_2^2 +  \frac{d^2\Delta^2}{\kappa^{\frac{2}{\kappa}}}$ instead of $(\sqrt{d}(M_2 + \Delta))^\frac{\kappa}{\kappa-1}$ from \citep{kornilov2023accelerated}.\color{black}

\noindent\textbf{Noise bounds.} In case of one-point oracle, while noise $\phi$ is “small”, i.e., 
\begin{equation}\label{eq:noise upper bound}
    \Delta  \leq \left(\frac{\kappa}{4} \right)^\frac1\kappa\frac{\varepsilon}{\sqrt{d}}
\end{equation}
 convergence rate is preserved. This bound on $\Delta$ is optimal in terms of  $\varepsilon$, see \citep{Lobanov_OPTIMA23, pasechnyuk2023upper, risteski2016algorithms}. 
 

%

\subsubsection{Other classes of the optimized functions}
\begin{remark}[Smooth objective]\label{rem:smooth}
    The estimates presented in Theorem~\ref{thm:SSTM col} can be improved by introducing a new assumption, namely the assumption that the objective function $f$ is $L$-smooth with $L>0$: ${\| \nabla f(y) - \nabla f(x)\|_2 \leq L \| y- x \|_2}$, $\forall x,y \in \R^d$. Using this assumption, we obtain the following value of the smoothing parameter $\tau = \sqrt{\varepsilon/L}$ \citep[see][the end of Section 4.1]{gasnikov2022randomized}. Thus, assuming smoothness and   convexity (As. \ref{as:f convex}) of the objective function and assuming symmetric noise (As. \ref{as:dist}), we obtain the following bounds for the sample complexity with independent oracle: 
    $$\widetilde{\cO}\left( \max\left\{\sqrt{\frac{L R^2}{\varepsilon}},  \frac{(\sqrt{d} R)^2}{b \cdot  \varepsilon^2}  \left( M_2^2  \lor \left(\frac{4}{\kappa} \right)^\frac2\kappa \frac{d L\Delta^2 }{\varepsilon}\right)\right\}\right)$$ and with Lipschitz oracle: $${\widetilde{\cO}\left( \max\left\{\sqrt{\frac{L R^2}{\varepsilon}},  \frac{d(M_2^2 + d \Delta^2/\kappa^{\frac{2}{\kappa}})R^2}{b \cdot  \varepsilon^2}  \right\} \right)}.$$  These rates match the sample's complexity for the full gradient coordinate-wise estimate. 
\end{remark}
\begin{remark}[Polyak–Lojasiewicz objective]
    The results of Theorem \ref{thm:SSTM col} can be extended to the case when the smooth objective function satisfies the Polyak–Lojasiewicz condition via restarts: let a function $f(x)$ is differentiable and there exists constant $\mu > 0$ s.t. $\forall x \in \R^d$ the following inequality holds ${\| \nabla f(x)} \|_2^2 \geq 2 \mu (f(x) - f(x^*))$. Then, assuming smoothness (see Remark \ref{rem:smooth}), Polyak–Lojasiewicz condition for the objective function and symmetric noise (As. \ref{as:dist}), we obtain the following bounds for the sample complexity with independent oracle: 
    $$\widetilde{\cO}\left( \max\left\{\frac{L}{\mu},  \frac{d L}{b \mu^2  \varepsilon}  \left( M_2^2  \lor \left(\frac{4}{\kappa} \right)^\frac2\kappa \frac{d L\Delta^2 }{\varepsilon}\right)\right\}\right)$$ and with Lipschitz oracle:
    $${\widetilde{\cO}\left( \max\left\{\frac{L}{\mu},  \frac{d L(M_2^2 + d \Delta^2/\kappa^{\frac{2}{\kappa}})}{b \mu^2  \varepsilon}  \right\} \right)}.$$ 
    See Appendix \ref{par:Restarts} for more details.
\end{remark}

\subsection{Our \algname{ZO-clipped-med-SMD} for constrained problems}\label{sec:SMD} 
We propose our novel \algname{ZO-clipped-med-SMD} to minimize functions on a convex compact $Q \subset \R^d$. We use unbatched median estimation \eqref{eq:med estimation}  in the zeroth-order algorithm \algname{ZO-clipped-SMD} from \citep{kornilov2023gradient}, which is based on \algname{Mirror Gradient Descent}. The pseudocode is presented in Algorithm \ref{alg:SMD}. 

We define $1$-strongly convex w.r.t. $\ell_p-$norm and differentiable prox-function $\Psi_p$, its convex (Fenchel) conjugate and Bregman divergence, respectively, as
\begin{eqnarray}
    \Psi_p^*(y) &=& \sup\limits_{x \in \R^d} \{\langle x,y \rangle - \Psi_p(x) \}, \notag \\
    V_{\Psi_p}(y,x) &=& \Psi_p(y) - \Psi_p(x) - \langle \nabla \Psi_p(x), y -x \rangle.\notag
\end{eqnarray}
\vspace{-15pt}
\begin{algorithm}
\caption{ \algname{ZO-clipped-med-SMD}  }\label{alg:SMD}
\begin{algorithmic}[1]
\REQUIRE{Number of iterations $K$, median size $m$, stepsize $\nu$, prox-function $\Psi_{p}$, smoothing parameter $\tau$, clipping level $\lambda$.}
    \STATE $x^0 = \arg\min\limits_{x \in Q} \Psi_{p}(x)$.
    
    \FOR{$k = 0, 1, \dots ,K-1$}  
        \STATE Sample $\mathbf{e}$ from $U(S_1(0))$ and sequence $\{\xi\}$.
        \STATE $g_{med}^{k+1} = \med^{m} (x^{k}, \mathbf{e}, \{\xi\}).$
        
        \STATE  $y^{k+1}  = \nabla(\Psi_{p}^*) (\nabla \Psi_{p}(x^k) - \nu \cdot \clip_q( g_{med}^{k+1}, \lambda)).$ 
        \STATE $x^{k+1}  = \arg \min\limits_{x \in Q} V_{\Psi_{p}}(x, y^{k+1}).$

    \ENDFOR

    \ENSURE $ \overline{x}^K := \frac{1}{K}\sum\limits_{k=0}^{K}x^k$
\end{algorithmic}
\end{algorithm}
\begin{theorem}[Convergence of \algname{ZO-clipped-med-SMD}] \label{thm:SMD col}
    Consider convex (As. \ref{as:f convex}) and $M_2$-Lipschitz (As. \ref{as:Lipshcitz}) function $f$ on $Q + B_{2\tau}(0)$ with two-point oracle corrupted by noise under As. \ref{as:dist} with $\kappa > 0$. 
    
    To achieve function accuracy $\varepsilon$, i.e., $f(\overline{x}^K ) - f(x^*) \leq \varepsilon$ with probability at least $1-\beta$  via \algname{ZO-clipped-med-SMD} with median size $m = \frac{2}{\kappa} + 1$, clipping level $\lambda = \sigma a_q \sqrt{K}$, stepsize $\nu = \frac{D_{\Psi_p}}{\lambda}$, diameter $D_{\Psi_p}^2 \eqdef 2 \sup\limits_{x,y \in Q}  V_{\Psi_{p}}(x,y)$, prox-function $\Psi_p$ and $\tau = \frac{\varepsilon}{4M_2}$, the number of iterations $K$ must be
    \begin{eqnarray} \label{eq:SMD conv iter num lip}
        \widetilde{\cO}\left(   \frac{(\sqrt{d} M_2 a_q D_{\Psi_{p}})^2}{  \varepsilon^2}  \left( 1  \lor \left(\frac{4}{\kappa} \right)^\frac2\kappa \frac{d\Delta^2}{\varepsilon^2}\right)\right), \notag \\
        \widetilde{\cO}\left(   \frac{d (M^2_2 + d \Delta^2/\kappa^{\frac{2}{\kappa}}) a^2_q D^2_{\Psi_{p}}}{ \varepsilon^2 }\right)
    \end{eqnarray}
     for \textbf{independent} and \textbf{Lipschitz} oracle, respectively. \!\!Each iteration uses $(2m\!+\!1)$ oracle calls.
\end{theorem}
The proof is similar to the proof of Theorem \ref{thm:SSTM col} and located in Appendix \ref{par: conv thms}.

\textbf{Optimality.} Bounds \eqref{eq:SMD conv iter num lip} match  optimal in terms of $\varepsilon$ bounds for stochastic non-smooth optimization on convex compact with the finite variance \citep{vural2022mirror}. The upper bound for $\Delta$ under which convergence rate is preserved matches the unconstrained case \eqref{eq:noise upper bound}.

\noindent\textbf{Acceleration via restarts.} For $\mu$-strongly-convex functions with Lipschitz oracle or independent oracle with small noise, we apply the restarted versions of our \algname{ZO-clipped-med-SMD} and \algname{ZO-clipped-med-SSTM}. Algorithms and results are located in Appendix~\ref{par:Restarts}.

\section{Application to the multi-armed bandit problem with heavy tails}
\label{sec:MAB}
In this section, we present our novel \algname{Clipped-INF-med-SMD} algorithm for multi-armed bandit (MAB) problem with heavy-tailed rewards.

\textbf{Introduction.} The stochastic MAB problem \citep{lattimore2020bandit} can be formulated as follows: an agent at each time step $t = 1,\dots, T$ chooses an action $A_t$ from a given action set $\mathcal{A} = (a_1, \dots, a_n)$ and suffers stochastic loss. For each action $a_i$, there exists a probability density function for losses $\mathbf{p}(a_i)$, and an agent doesn't know them in advance. An agent can observe losses only for one action at each step, namely, the one it chooses. At each round $t$, when action $a_i$ is chosen (i.e. $A_t = a_i$), stochastic loss $\mu_{A_t}+\xi_{A_t, t}$ sampled from $\mathbf{p}(a_i)$ independently. Agent's goal is to minimize \textit{average regret}:
\begin{equation*}\vspace{-0.1cm}
    \EE[\mathcal{R}_T] =  \sum_{t=1}^T \left [ \mu_{A_t} - \mu^* \right ], \quad \mu^* = \min_{a_i \in \mathcal{A}} \mu_i. \vspace{-0.1cm}
\end{equation*}
One of the main approaches for solving the MAB problem is to use reduction to the online convex optimization problem \citep{hazan2016introduction, orabona2019modern}. Consider stochastic linear loss functions $l_t(x_t) = \langle \mu+\xi_t, x_t \rangle$, with noise $\xi_t$ and unknown fixed vector of expected losses $\mu \in \R^d$. The decision variable $x_t \in  \triangle^d_+$ can be viewed as  player's mixed strategy (probability distribution over arms), which they use to sample arms to minimize expected regret 
$$
\EE[\mathcal{R}_T (u)] =  \EE \left [ \sum_{t=1}^T l_t(x_t) - \min_{u \in \triangle^d_+} \left (\sum_{t=1}^T l_t(u) \right )\right ].
$$
The player observes only sampled losses for the chosen arm, i.e., the (sub)gradient $g (x) \in \partial l(x)$ is not observed in the MAB setting, and one must use an inexact oracle instead. 

\textbf{Related works.} Bandits with heavy tails were introduced in \citep{liu2011multi, bubeck2013bandits}. The heavy noise assumption usually requires the existence of $\kappa \in (1, 2]$, such that $\EE[\|\mu+\xi_t\|^{\kappa}] \leq \sigma^{\kappa}$ {(in this work, we use different Assumption~\ref{as:dist} with $\kappa >0$).  In \citep{bubeck2013bandits}, the authors provide lower bounds on regret $\Omega \left ( \sigma d^{\frac{\kappa - 1}{\kappa}} T^{\frac{1}{\kappa}}   \right )$ and nearly optimal algorithmic scheme called Robust UCB. Recently, a few optimal algorithms were proposed \citep{lee2020optimal, zimmert2019optimal, huang2022adaptive, dorn2023} with regret bound $\tilde{O} \left ( \sigma d^{\frac{\kappa - 1}{\kappa}} T^{\frac{1}{\kappa}}   \right )$. \algname{HTINF} \citep{huang2022adaptive} is an INF-type algorithm with a specific pruning procedure. Algorithm \algname{1/2-Tsallis} \citep{zimmert2019optimal} is similar to \algname{HTINF}. \algname{INF-clip} \citep{dorn2023} employs a clipping mechanism instead of pruning, it clips rewards at the initial stage of the estimator construction process, prior to applying importance weighting. The main drawback of this procedure that the importance weighting procedure can artificially produce a burst in the gradient estimator.  Finally, \algname{APE} \citep{lee2020optimal} is a perturbation-based exploration strategy that uses a p-robust mean estimator. Its algorithmic scheme is UCB-type and is very different from our algorithm. 
}

\textbf{Our approach.} We assume that noise $\xi_t$ satisfy Assumption~\ref{as:dist} for some $\kappa > 0$. We construct our \algname{Clipped-INF-med-SMD} (Algorithm \ref{alg:bandits}) based on Online Mirror Descent, but in case of symmetric noise we can improve regret upper bounds and make it $\tilde{O} (\sqrt{dT})$ which is optimal compared to the lower bound $\Omega(\sqrt{dT})$ for stochastic MAB with the bounded variance of losses.  
In our algorithm, we use an importance-weighted estimator: $$\hat{g}_{t,i} = \begin{cases} \frac{g_{t,i}}{x_{t, i}} & \text{if } i = A_t\\ 0 & \text{otherwise} \end{cases},$$ where $A_t$ is the index of the chosen (at round $t$) arm.  This estimator is unbiased, i.e. $\EE_{x_t}[\hat{g}_t] = g_t$. The main drawback of this estimator is that, in the case of small $x_{t,i}$, the value of $\hat{g}_{t,i}$ can be arbitrarily large. When the noise $g_t - \mu$ has heavy tails (i.e., $\|g_t - \mu\|_{\infty}$ can be large with high probability), this drawback can be amplified. That is why we use robust median estimation with  clipping.  

\begin{algorithm}[!h]
\caption{ \algname{Clipped-INF-med-SMD}  }\label{alg:bandits}
\begin{algorithmic}[1]
\REQUIRE{Time period $T$, median size $m$, stepsize $\nu$, prox-function $\Psi_{p}$,  clipping level $\lambda$.}
    \STATE $x_0 = \arg\min\limits_{x \in  \triangle^d_+} \Psi_{p}(x)$.
    \STATE Set number of iterations $K = \left \lceil\frac{T-1}{2m+1}\right \rceil$.
    \FOR{$k = 0, 1, \dots ,K - 1$}  
        \STATE Draw $A_t$ for $2m+1$ times ($t=(2m+1)\cdot k+1,\dots,(2m+1)\cdot (k+1)$) with $P(A_t = i) = x_{k,i}$, $i=1,\dots, d$ and observe rewards $g_{t, A_t}$.
        \STATE For each observation, construct estimation $\hat{g}_{t,i} = \begin{cases} \frac{g_{t,i}}{x_{k,i}} & \text{if } i = A_t \\ 0 & \text{otherwise} \end{cases}, i=1,\dots, d.$

        \STATE \label{eq:SMD_get_oracle} $g_{med}^{k+1} = \texttt{Median}( \{\hat{g}_{t}\}_{t=(2m+1)\cdot k+1}^{(2m+1)\cdot(k+1)}). $
        \STATE  $y_{k+1}  = \nabla(\Psi_{p}^*) (\nabla \Psi_{p}(x_k) - \nu \cdot \clip_q\left(g_{med}^{k+1}, \lambda\right)). $
        \STATE $ x_{k+1}  = \arg \min\limits_{x \in  \triangle^d_+} V_{\Psi_{p}}(x, y_{k+1}).$

    \ENDFOR

\end{algorithmic}
\end{algorithm}
\begin{theorem} [Convergence of \algname{Clipped-INF-med-SMD}] 
    \label{avrregretest}
Consider MAB problem where the conditional probability density function for each loss satisfies As.~\ref{as:dist} with $\Delta, \kappa > 0,$ and $ \|\mu\|_{\infty} \leq R$. Then, for the period $T$,  the sequence $\{x_t\}_{t=1}^T$ generated  by \algname{Clipped-INF-med-SMD} with parameters $m = \frac{2}{\kappa}+1$, $\tau = \sqrt{d}$, $\nu = \frac{\sqrt{(2m+1)}}{\sqrt{T (36 c^2 + 2R^2)}}$, $\lambda = \sqrt{T}$ and prox-function $\Psi_1(x) = \psi(x) \eqdef 2\left (1 - \sum_{i=1}^d x_i^{1/2} \right)$  satisfies for all $ u \in \Delta_+^d$
\begin{equation}
    \EE \left [\mathcal{R}_T (u) \right ]  \leq \sqrt{dT} \cdot (8c^2/\sqrt{d} + 4 \sqrt{ (2m+1)(18 c^2 + R^2)}),
\end{equation}
where $c^2 = \left ( 32 \ln d - 8 \right ) \cdot \left (8M_2^2 +  2\Delta^2 (2m+1)\left(\frac{4}{\kappa} \right)^{\frac2\kappa} \right )$.  Moreover, high probability bounds from  Theorem \ref{thm:SMD col} also hold.
\end{theorem}
The proof of Theorem \ref{avrregretest} is located in  Appendix~\ref{app:proof_of_th_3}.
\vspace{-0.25cm}
\section{Numerical Experiments}\label{sec:Exp}
\vspace{-0.1cm}
In this section, we demonstrate the superior performance of ours \algname{ZO-clipped-med-SSTM} and \algname{Clipped-INF-med-SMD} under heavy-tailed noise on experiments on syntactical and real-world data. Additional experiments and technical details are located in Appendix \ref{sec:app exp}. \textit{All data is taken from publicly available sources or synthesized following the experiments' descriptions.}
\vspace{-0.1cm}
\subsection{Multi-Armed Bandit}
\vspace{-0.1cm}
  {We compare our \algname{Clipped-INF-med-SMD} with popular SOTA algorithms tailored to handle MAB problem with heavy tails, namely, \algname{HTINF} and \algname{APE}.} We focus on an experiment involving only two available arms ($d=2$). Each arm $i$ generates random losses $g_{t, i} \sim \xi_t + \beta_{i}$. Parameters $\beta_{0} = 3, \beta_{1} = 3.5$ are fixed, and independent random variables $\xi_t$ have the same probability density $\text{p}_{\xi_t}\left(x\right) = \frac{1}{3 \cdot \left(1 + \left(\frac{x}{3}\right)^2\right) \cdot \pi}$. 

\noindent For all methods, we evaluate the distribution of expected regret and probability of picking the best arm over $100$ runs. The results are presented in Figure \ref{fig:galaxy4}.
\begin{figure*}[!h]
    \centering
\includegraphics[width=0.85\textwidth]{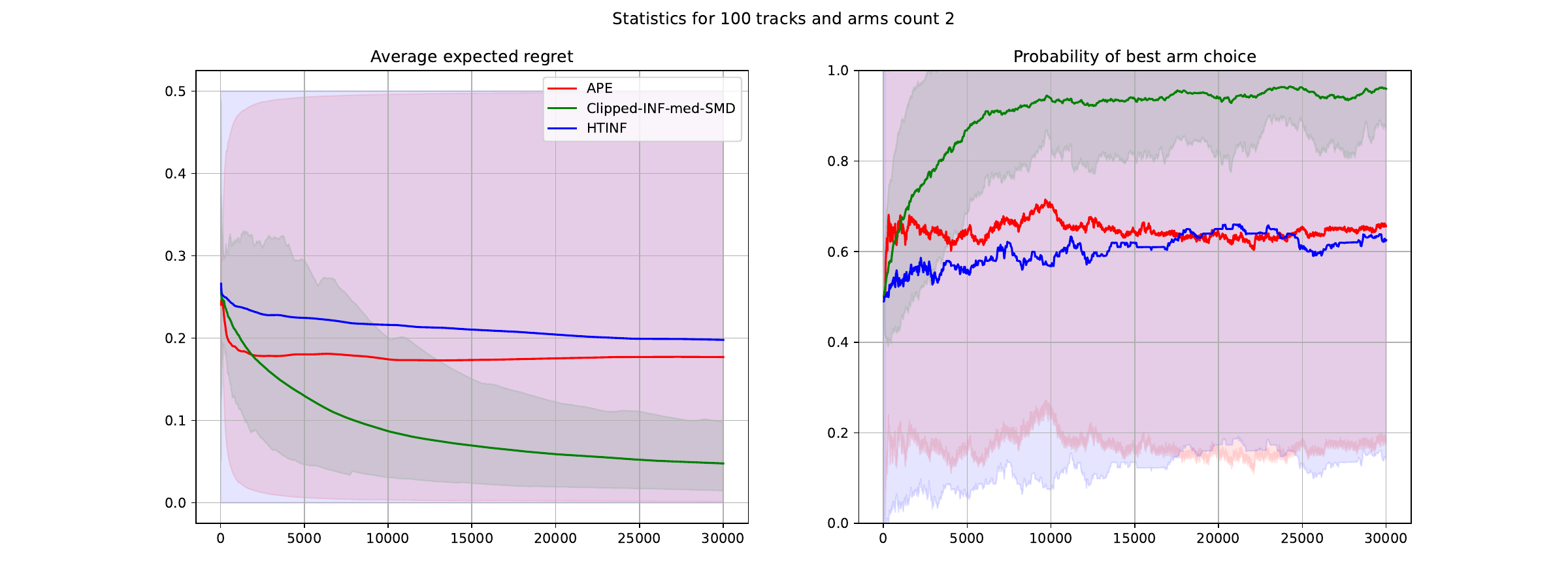}
    \caption{Average expected regret and probability of optimal arm picking mean for $100$ experiments and $30000$ samples with $0.95$ and $0.05$ percentiles for regret and ± std bounds for probabilities.}
    \label{fig:galaxy4}
\end{figure*}

As one can see from the graphs, \algname{HTINF} and \algname{APE} do not have convergence in probability, while our \algname{Clipped-INF-med-SMD} does, which confirms the efficiency of the proposed method.  In Appendix \ref{sec:bandit extra exp}, we provide technical details and additional experiments for different $\kappa$.
\vspace{-0.1cm}
\subsection{Cryptocurrency portfolio optimization}
\vspace{-0.1cm}
We choose cryptocurrency portfolio optimization problem for our \algname{Clipped-INF-med-SMD} real world application, since cryptocurrency pricing data is known by having heavy-tailed distribution. In our scenario, we have $n=9$ assets for investing. At step $t$, we choose assets’ distribution $x_{t,i} \in \Delta^n$ and then observe the whole income vector $r_{t,i}$ for each asset $i$. The main goal is to maximize total income $\max \mathbb{E} [\sum_{t = 1}^T \sum_{i = 1}^n r_{t,i} x_{t,i}]$ over a fixed time interval~$T$.

Portfolio selection has the full feedback for all assets, while, in standard bandits, we observe only one asset per step. We adjust our \algname{Clipped-INF-med-SMD} for the full feedback via calculating line 4 in Algorithm \ref{alg:bandits} for each asset $i$. As baselines, we use two strategies: hold ETH and the Efficient Frontier method \citep{markovitz} with maximal sharp ratio portfolio selected. For a dataset, we use open prices from Binance Spot for 2023.

The results are presented in Figure \ref{fig:crypto}. As one can see, the Efficient Frontier strategy can't efficiently perform on cryptocurrency assets, and our \algname{Clipped-INF-med-SMD} achieved higher performance than just holding the ETH strategy, so it can be applied for detecting potentially promising assets.

\begin{figure*}[!h]
    \centering
    \includegraphics[width=1\textwidth]{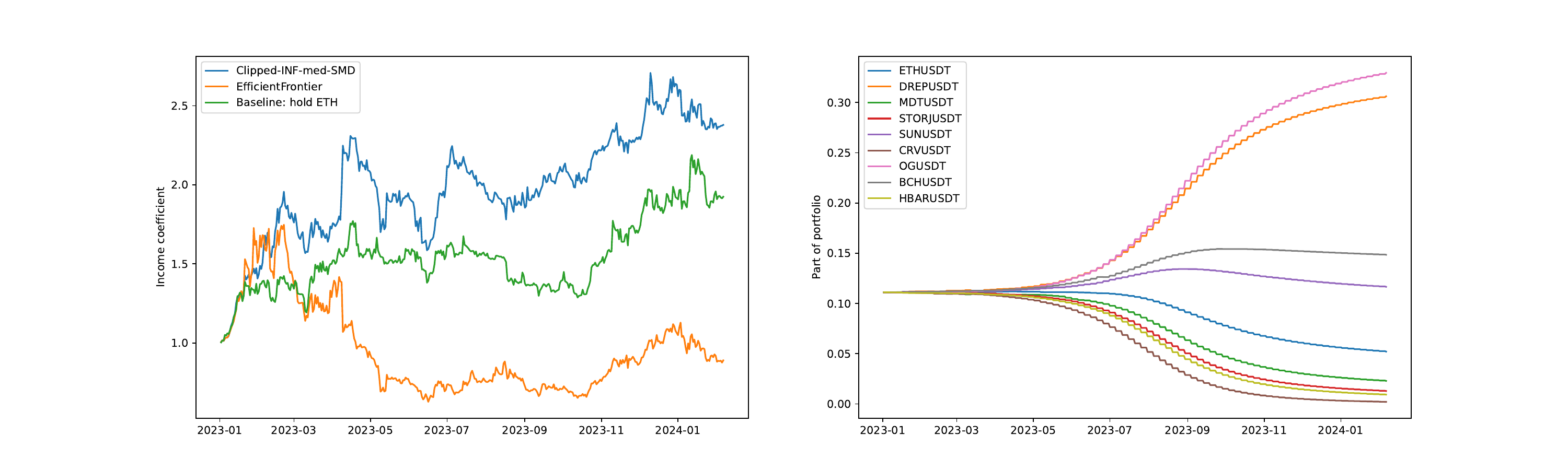}
    \caption{Strategies profit coefficient and \algname{Clipped-INF-med-SMD} assets distribution over 2023.}
    \label{fig:crypto}
\end{figure*}

\subsection{Zeroth-order optimization} 
To demonstrate the performance of our \algname{ZO-clipped-med-SSTM}, we follow \citep{kornilov2023accelerated} and conduct experiments on the following  problem:
\vspace{-0.15cm}  $$ \min_{x\in \mathbb{R}^{d}} \{ \| Ax - b \|_2 +  \langle \xi, x \rangle\}, \vspace{-0.15cm}$$
where  $\xi$ is a random vector with independent components sampled from the symmetric Levy $\alpha$-stable distribution with different $\alpha = 0.75, 1.0, 1.25, 1.5$,  $A \in \mathbb{R}^{l\times d}$, $b \in \mathbb{R}^{l}$.  Note, that $\alpha$ has the same meaning as $\kappa$, because this distribution asymptotic behavior is $f(x)\sim \frac{1}{|x|^{1+\alpha}}$ for $\alpha < 2$.

The best median size for our \algname{ZO-clipped-med-SSTM} is $m = 2$. We compare it with the median size $m = 0$ which is basically \algname{ZO-clipped-SSTM}.  We additionally compare our algorithm with \algname{ZO-clipped-SGD} from \citep{kornilov2023accelerated} and \algname{ZO-clipped-med-SGD} ~--- version of \algname{ZO-clipped-SGD} with gradient estimation step replaced with median clipping version from our work. In Appendix \ref{sec:ZO extra exp}, we provide technical details about hyperparameters.

\subsubsection{Comparison with the competitors} To make a fair comparison of the methods' performance, we conduct the experiment with $\kappa = 1$ over $15$ launches and present the results in Figure~\ref{fig:zo-clipped-sstm_comparison_synth}. 
\begin{figure}[!h]
    \centering
    \includegraphics[width=0.6\textwidth]{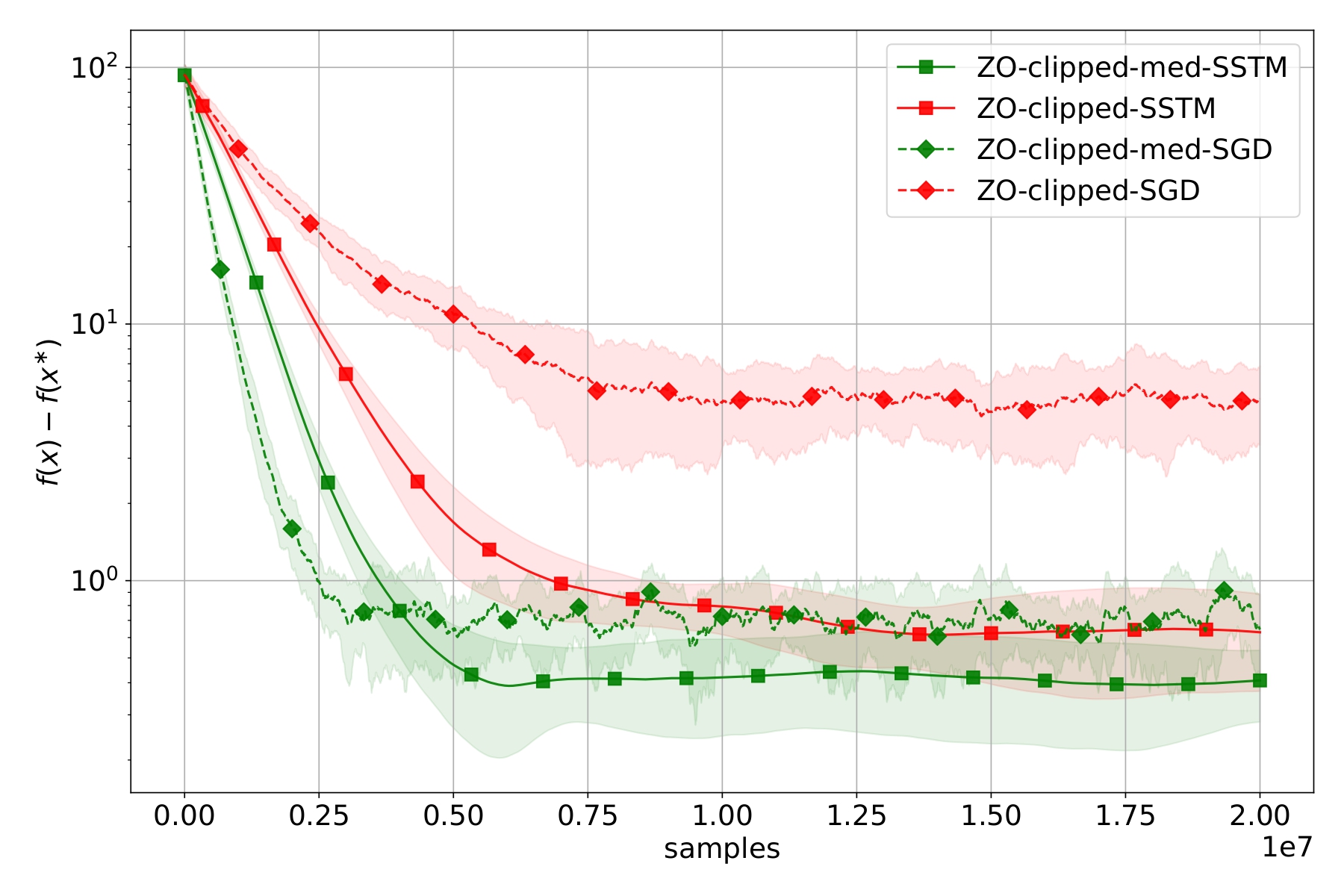}
    \caption{Convergence of our \algname{ZO-clipped-SSTM} and \algname{ZO-clipped-med-SSTM}, \algname{ZO-clipped-SGD}, \algname{ZO-clipped-med-SGD} over $15$ launches.}
    \label{fig:zo-clipped-sstm_comparison_synth}
\end{figure}
The $\kappa$ dependency results over $3$ launches are presented in Figure \ref{fig:zo-clipped-med-sstm_comparison_synth_different_kappa}. 

The green lines on the graphs represent algorithms with median clipping. We can see that for extremely noised data $\kappa \leq 1$, our median clipping-based methods significantly outperform non-median versions. While, for standard heavy-tailed noise $\kappa > 1$, our methods do not lose to other competitors. 
\begin{figure*}[!h]
    \centering
    \includegraphics[scale=0.17]{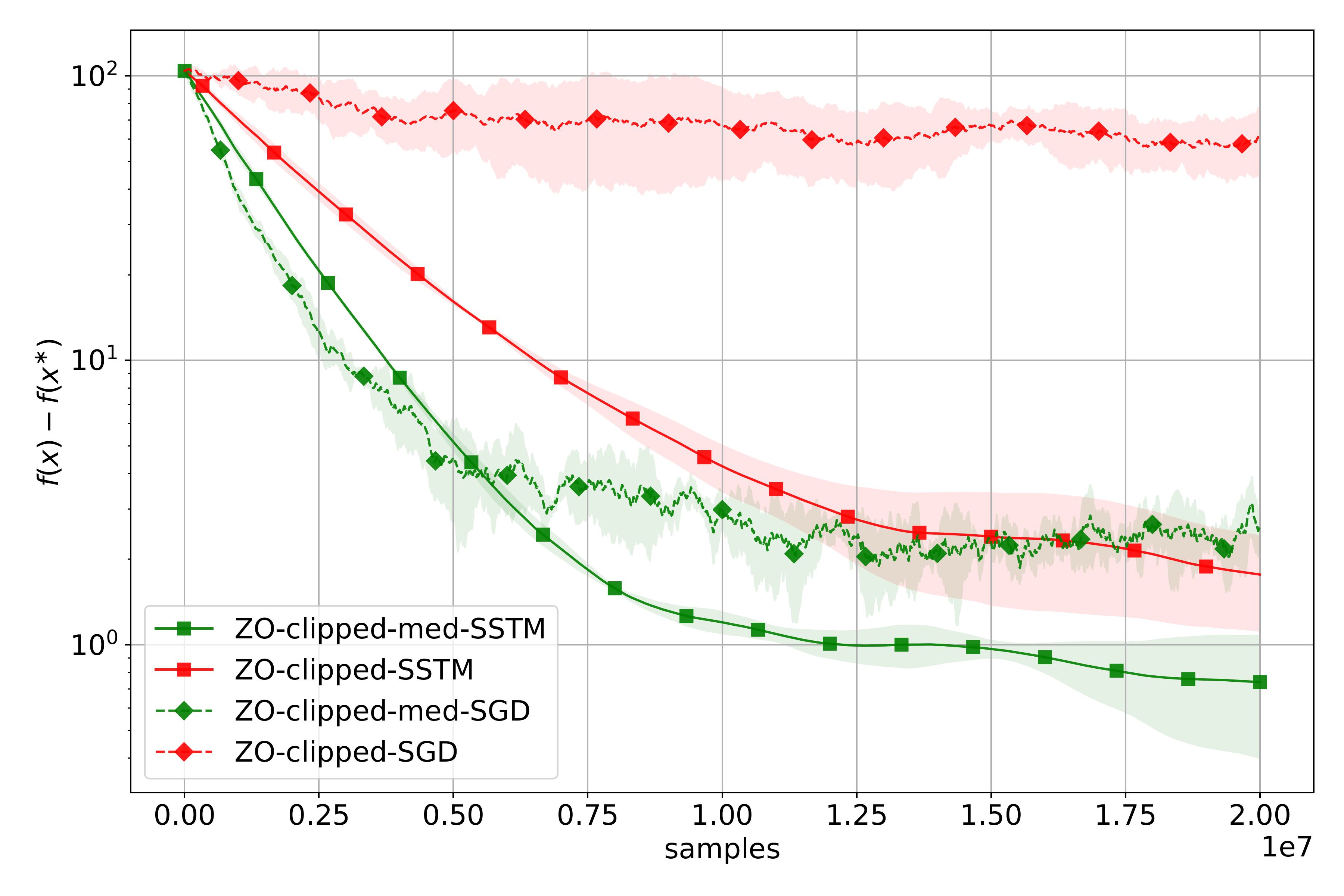}
    \includegraphics[scale=0.17]{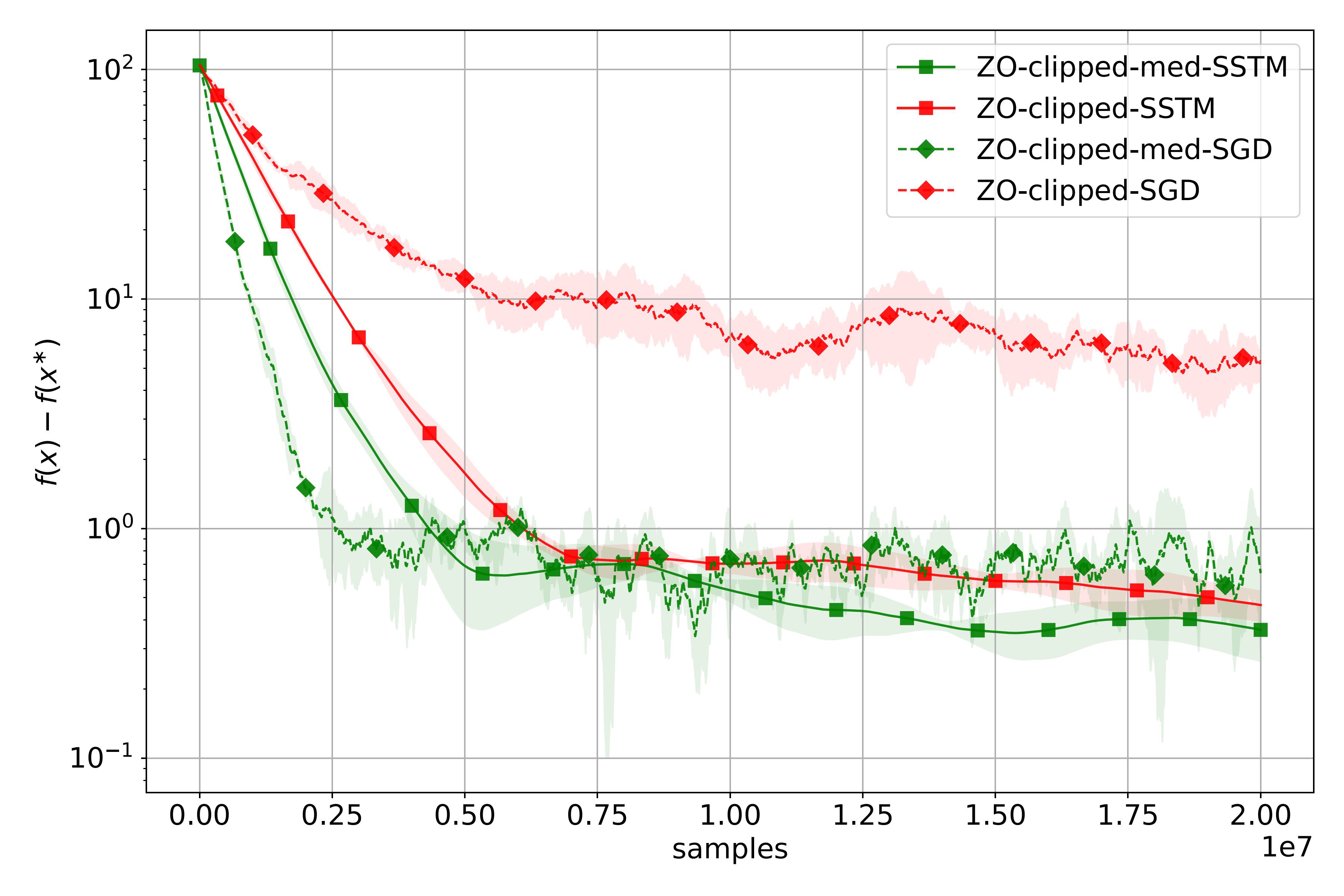}
    \includegraphics[scale=0.17]{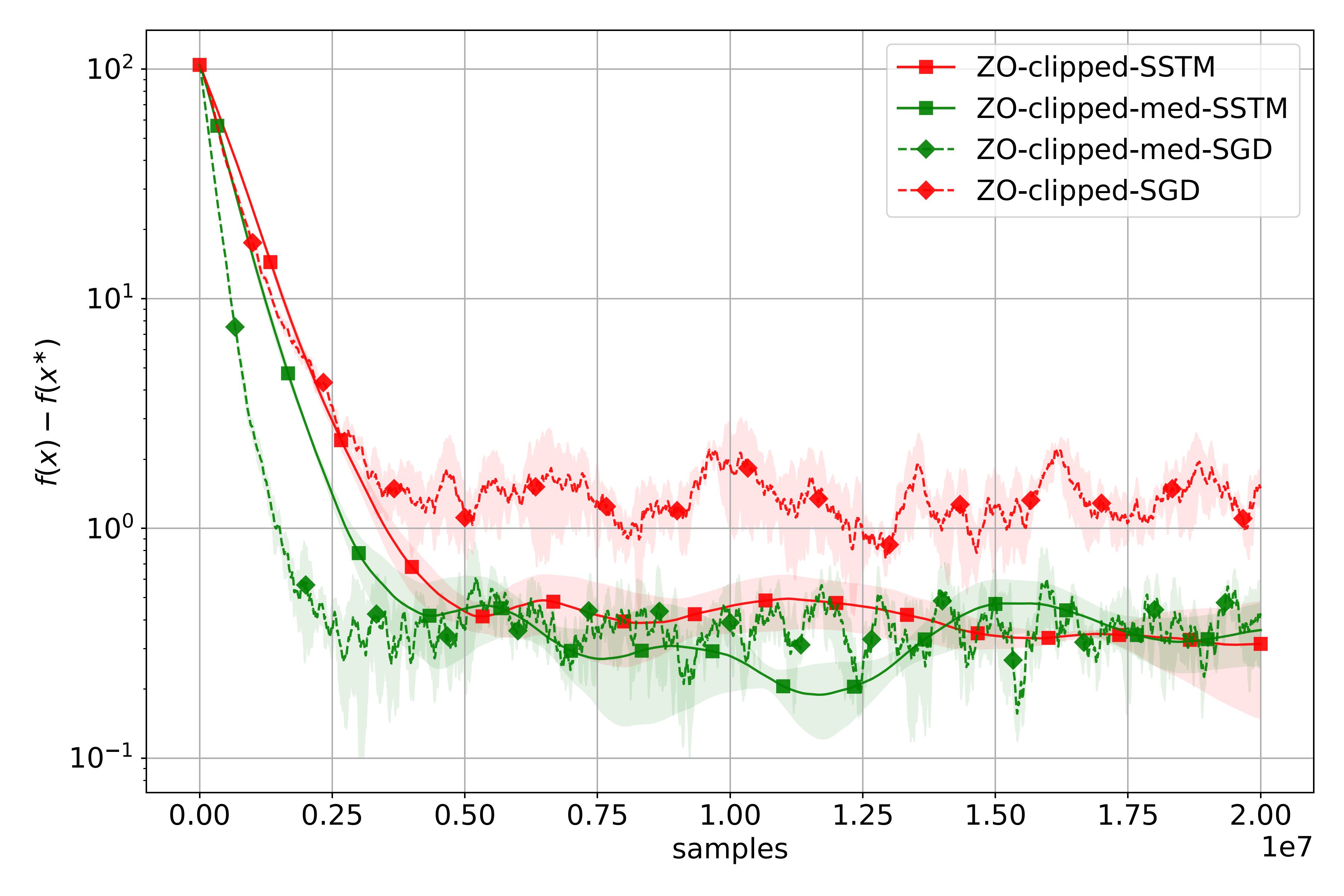}
    \includegraphics[scale=0.17]{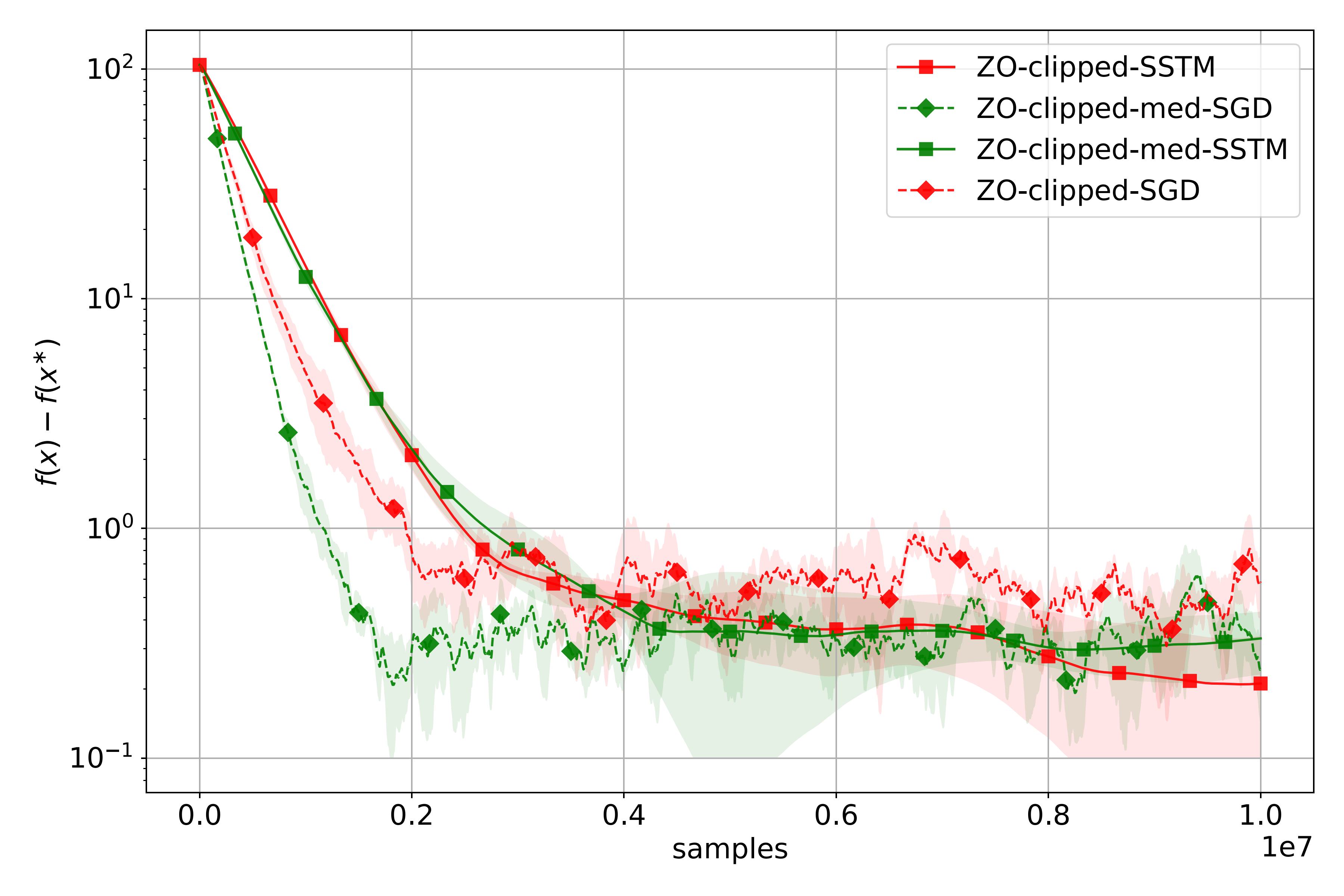}
    \caption{Convergence of our \algname{ZO-clipped-SSTM} and \algname{ZO-clipped-med-SSTM}, \algname{ZO-clipped-SGD}, \algname{ZO-clipped-med-SGD} in terms of a gap function w.r.t. the number of used samples from the dataset for different $\alpha = \kappa$ parameters (left-to-right and top-to-bottom: 0.75, 1.0, 1.25, 1.5).}
    \label{fig:zo-clipped-med-sstm_comparison_synth_different_kappa}
\end{figure*}
\vspace{-0.1cm}

\subsubsection{Asymmetric noise} \label{sec: sym and asym} To check the dependence on the addition of an asymmetric part to the noise, we replace the noise $\xi$ with $\xi = w * \xi_1 + (1-w) * |\xi_2|$ with $\xi_1$ drawn from a symmetric Levy $\alpha$-stable distribution with $\alpha = 1.0$ and $\xi_2$ being a random vector with independent components sampled from
\begin{itemize}
    \item the same distribution;
    \item standard normal distribution.
\end{itemize}
For $w$, we consider $0.9$ (the weight of symmetric noise is bigger) and $0.5$ (equal impact). We take a component-wise absolute values of $\xi_2$, which makes $w$ a mix of symmetric and asymmetric noise. The results are presented in Figures \ref{fig:zo-clipped-med-sstm_comparison_Levy_assym_mix} (Levy noise) and \ref{fig:zo-clipped-med-sstm_comparison_normal_assym_mix} (normal noise).

\begin{figure*}[!h]
    \centering
    \includegraphics[scale=0.06]{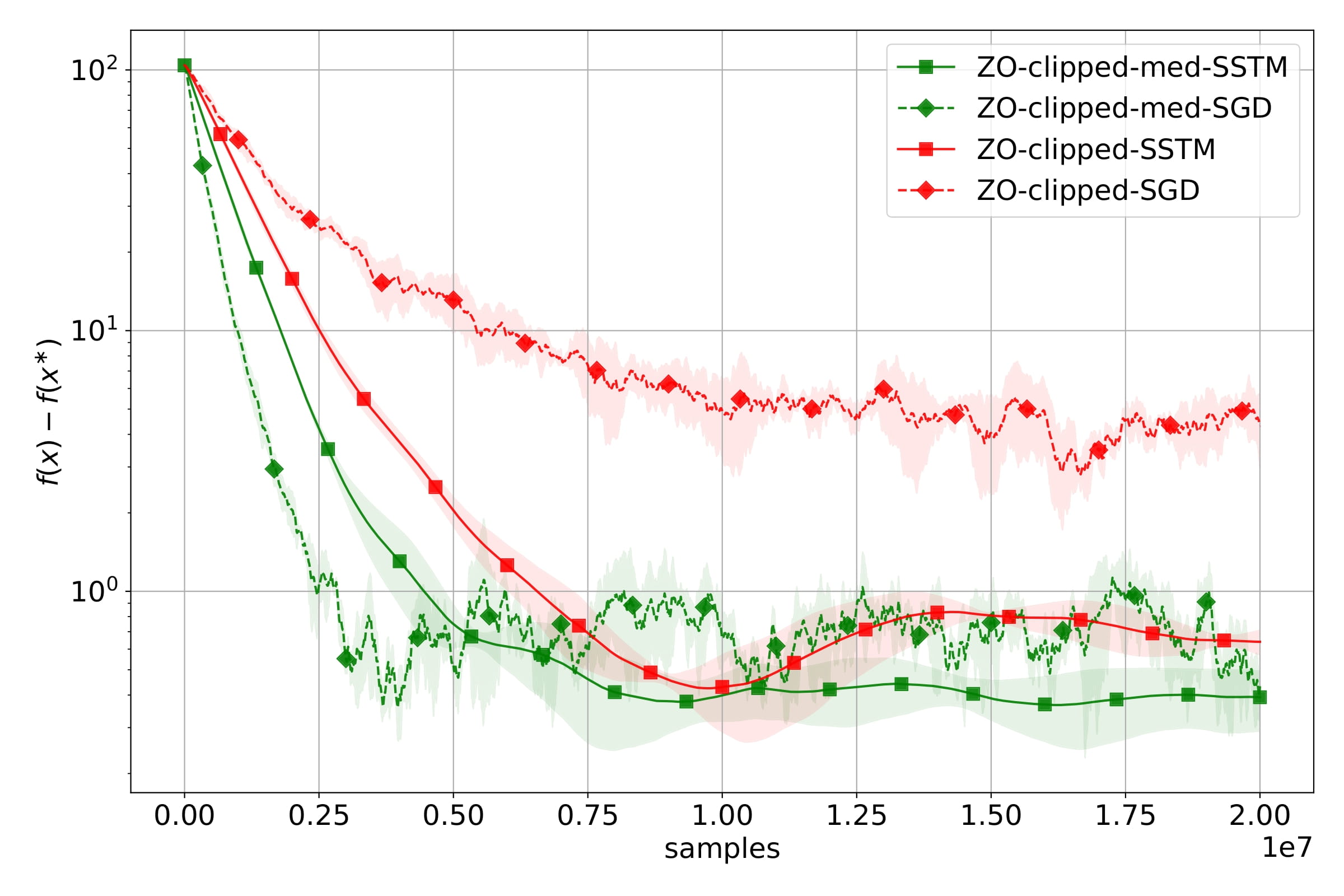}
    \includegraphics[scale=0.06]{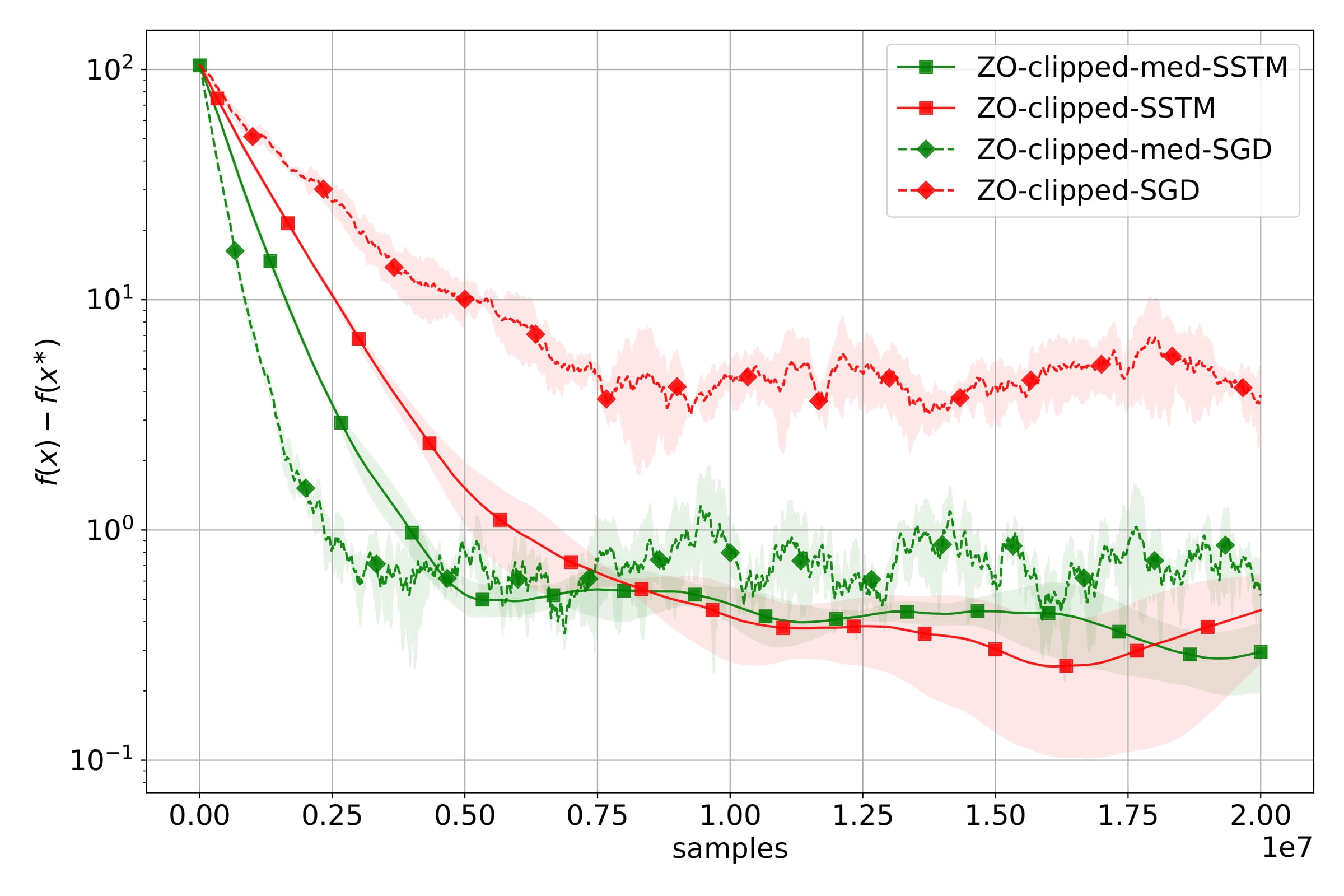}
    \caption{Convergence of our \algname{ZO-clipped-SSTM} and \algname{ZO-clipped-med-SSTM}, \algname{ZO-clipped-SGD}, \algname{ZO-clipped-med-SGD} with asymmetric Levy noise addition with weight of symmetric part of $0.9$ and $0.5$ on left and right, respectively.}
    \label{fig:zo-clipped-med-sstm_comparison_Levy_assym_mix}
\end{figure*}

\begin{figure*}[!h]
    \centering
    \includegraphics[scale=0.06]{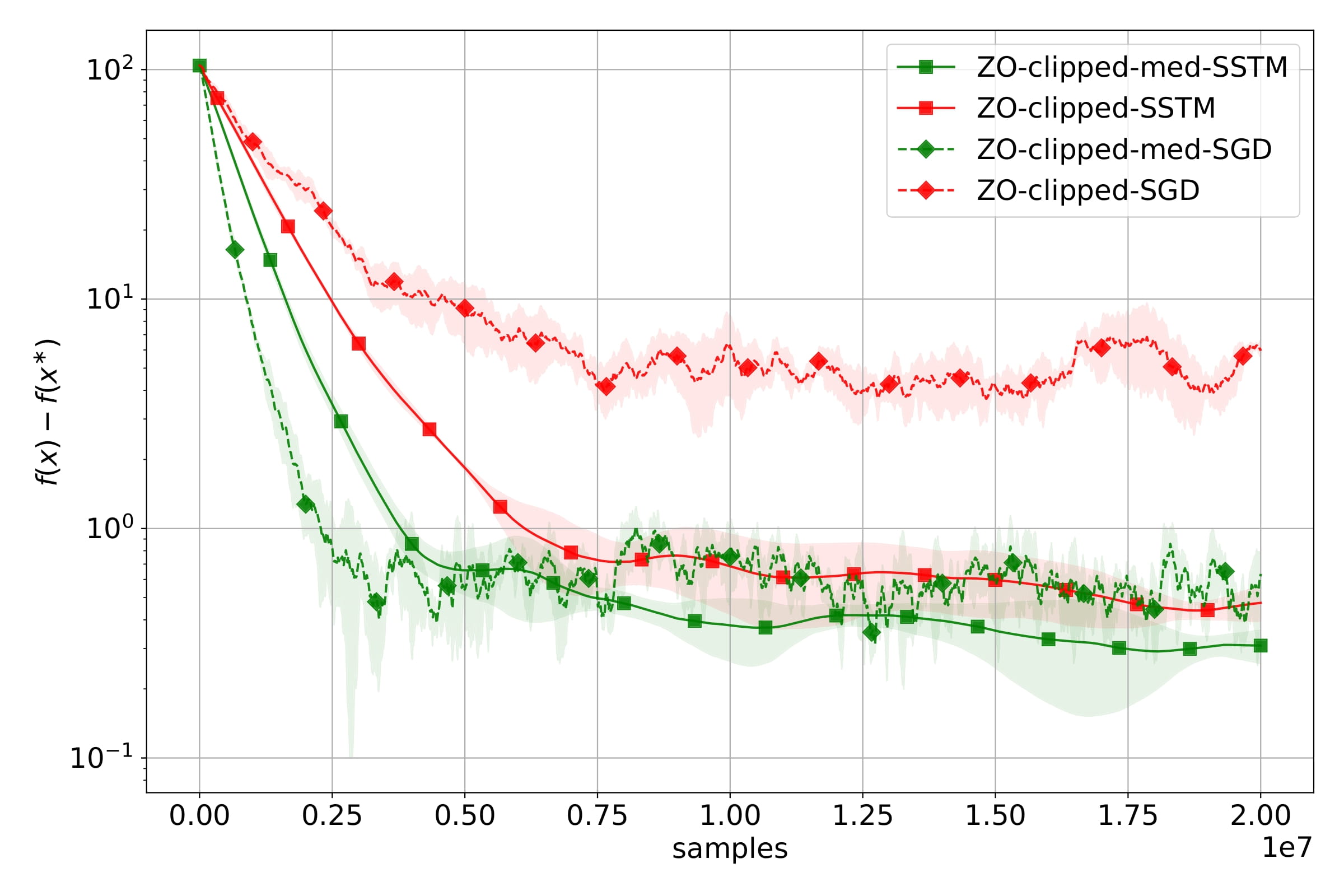}
    \includegraphics[scale=0.06]{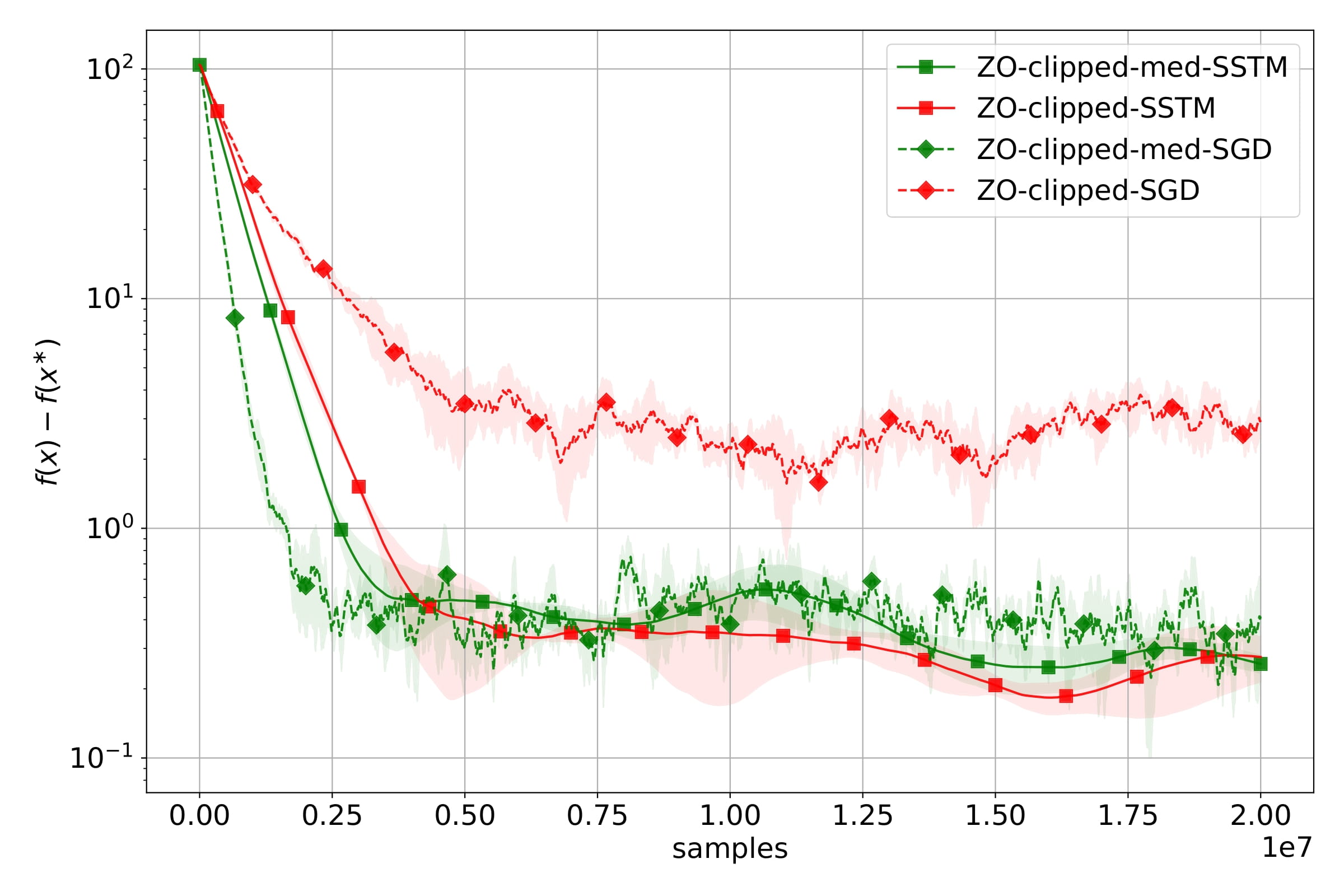}
    \caption{Convergence of our \algname{ZO-clipped-SSTM} and \algname{ZO-clipped-med-SSTM}, \algname{ZO-clipped-SGD}, \algname{ZO-clipped-med-SGD} with asymmetric normal noise addition with weight of symmetric part of $0.9$ and $0.5$ on left and right, respectively.}
    \label{fig:zo-clipped-med-sstm_comparison_normal_assym_mix}
\end{figure*}

\textbf{Tuning of the median size $m.$} In experiments with both bandits and ZO methods, we grid search the median size $m$ among the range [3,5,7]. We noticed that unlike the choice of continuous the clipping level, the choice of the discrete median size only slightly affects the convergence and does not require careful fine-tuning. This range is enough to find an optimal median size for optimal convergence.\color{black}
\vspace{-0.3cm}
\section{Discussion}\label{sec:limitations}

\subsection{Limitations}
\vspace{-0.1cm}
\textbf{Symmetric noise.} The assumption of the symmetric noise can be seen as a limitation from a practical point of view. It is indeed the case, but we argue that it is not as severe as it looks. 
A common strategy to solve a general optimization problem is to run several algorithms in a competitive manner to see which performs better in practice. This approach is implemented in industrial solvers such as Gurobi. Thus, if we have different algorithms, each suited to its own conditions, we can simply test to see which one is faster for our particular case. In this scenario, we want a set of algorithms, each designed for its specific case. Our algorithm can serve as one of the options in such mix, since it provides considerable acceleration in a significant number of noise cases.  Moreover, in experiments with non-symmetric noises ($\S$\ref{sec: sym and asym}), our methods do not lose to the baselines. Hence, running our methods ends up with either typical convergence rates or faster rates for symmetric noises. 

One of the most standard and widely applied assumptions in MAB Literature is assumption of normal normality. Normal distributions belong to symmetric noise, which we consider. Hence, our work has huge potential impact for the field.       

\textbf{Do we need to know the explicit $\kappa$?} In our Theorems \ref{thm:SSTM col}, \ref{thm:SMD col}, \ref{avrregretest}, parameter $\kappa$ is required to set optimal median size  $m = \frac{2}{\kappa} + 1$. However, for the most common cases $\kappa$ is at least $1$ (i.e. expectation exists), hence we could take median size $m = 3$. In case when parameter $\kappa \to 0$, we leave the construction of an adaptive scheme \citep{huang2022adaptive} for future work. In practice, the choice of $m$ can be limited to a small, discrete range. 
\color{black}


\subsection{Comparison with the previous works}
Unlike the baselines  \algname{ZO-clipped-SSTM} \citep{kornilov2023accelerated} and \algname{APE} \citep{lee2020optimal}, \algname{HTINF} \citep{huang2022adaptive} with simple clipping and general heavy-tailed noise assumption $\kappa \in (1,2]$, our Algorithms \ref{alg:clipped-SSTM}, \ref{alg:SMD}, \ref{alg:bandits} with median clipping can work under extremely heavy-tailed noises $\kappa \leq 1$. For any $\kappa>0$, iterative complexity of our methods remains as if noise had bounded variance, namely, $\tilde{O}(d^2\varepsilon^{-2})$ iterations to achieve function accuracy or average regret$~\varepsilon$. In contrast, the best-known baselines' rates $\tilde{O}((\sqrt{d}\varepsilon^{-1})^{\frac{\kappa}{\kappa - 1}})$ deteriorate depending on $\kappa$. However, such breaking results can be guaranteed only for symmetric noises, which is not as serious limitation as it seems. We show that, for asymmetric noises, our methods in practice are competitive as well and perform at the same level as the baselines ($\S$\ref{sec: sym and asym}). 

 \color{black}
 \vspace{-0.2cm}

\section{Acknowledgements}
This work was supported by a grant, provided by the Ministry of Economic Development of the Russian Federation in accordance with the subsidy agreement (agreement identifier 000000C313925P4G0002) and the agreement with the Ivannikov Institute for System Programming of the Russian Academy of Sciences dated June 20, 2025 No. 139-15-2025-011.
\newpage


\bibliography{refs}
\bibliographystyle{apalike} 

\newpage
\appendix

\onecolumn
\section{Remarks about the assumption on the noise} \label{par: remarks on as}

In this section, we discuss our novel noise Assumption  \ref{as:dist}. We provide comparison with previous works (Remark \ref{rem:diff_as}), standard examples (Remark \ref{rem:examples}) and explain the roles of parameters (Remark \ref{rmk:role}).  
\begin{remark}[Comparison with previous assumptions]\label{rem:diff_as}
    In works \citep{Dvinskikh_2022, kornilov2023accelerated}, different assumption on Lipschitz noise is considered. For any realization of $\xi$,  the function $f(x, \xi)$ is $M'_2(\xi)$-Lipschitz, i.e.,
    \begin{equation}\label{eq:prev_noise_as}
        |f(x,\xi) - f(y, \xi)| \leq M'_2(\xi)\|x-y\|_2, \quad \forall x, y \in Q
    \end{equation}
    and $M'_2(\xi)^\kappa$ has bounded $\kappa$-th moment ($\kappa$ > 1), i.e., $[M'_2]^{\kappa} \eqdef \EE_\xi[M'_2(\xi)^\kappa] < \infty$.

    We emphasize that if Assumption~\ref{as:dist} holds with $\kappa$ then one can find $M'_2(\xi, x, y)$ such that \eqref{eq:prev_noise_as} holds for any $1 < \kappa' < \kappa$ with $M'_2 = O(M_2 + \Delta)$, where constant in $O(\cdot)$ depends only on $\kappa'$.    
\end{remark}
\begin{proof}
    Let noise $\phi(\xi| x,y)$ satisfies Assumption~\ref{as:dist} with Lipschitz oracle and $\kappa > 1$, then it holds
\begin{eqnarray}
    |f(x, \xi) - f(y, \xi)| &=& | f(x) - f(y) + \phi(\xi| x,y) | \notag \\
    &\leq& | f(x) - f(y)| + | \phi(\xi| x,y)| \notag \\
    &\overset{\text{As~\ref{as:Lipshcitz}}}{\leq}& M_2 \|x - y\|_2 + \frac{|\phi(\xi| x,y)|}{\|x-y\|_2}\|x-y\|_2. \notag
\end{eqnarray}
Let us denote $M'_2(\xi, x, y) \eqdef M_2 + \frac{| \phi(\xi| x,y)|}{\|x-y\|_2}$ and show that for any $1 < \kappa' < \kappa $ random variable $M'_2(\xi, x, y)$ has bounded $\kappa'$-th moment which doesn't depend on $x,y$. We notice that 
\begin{eqnarray*}
    \EE_\xi[|\phi(\xi| x,y)|^{\kappa'}] &=& \int \limits_{-\infty}^{+\infty} |u|^{\kappa'}p(u|x,y) du \\
    &\leq& \int \limits_{-\infty}^{+\infty} \frac{|u| ^{\kappa'} \gamma^\kappa|B(x,y)|^\kappa}{|B(x,y)|^{1+\kappa} + |u|^{1+\kappa}} du.
\end{eqnarray*}
After substitution $t = \nicefrac{u}{|B(x,y)|}$, we get
\begin{eqnarray*}
    \EE_\xi[|\phi(\xi| x,y)|^{\kappa'}] &\leq& \frac{\gamma^\kappa|B(x,y)|^{\kappa}}{|B(x,y)|^{\kappa - \kappa'}}\int \limits_{0}^{+\infty} \frac{|t| ^{\kappa'}}{1 + |t|^{1+\kappa}}  dt\\
    &\overset{\eqref{eq:noise_lip}}{\leq}& \gamma^{\kappa - \kappa'}\Delta^{\kappa'} \|x-y\|_2^{\kappa'}\int \limits_{0}^{+\infty} \frac{|t| ^{\kappa'}}{1 + |t|^{1+\kappa}}  dt.\\
\end{eqnarray*}
Integral $I(\kappa') = \int \limits_{0}^{+\infty} \frac{\gamma^{\kappa - \kappa'} |t| ^{\kappa'} dt}{1 + |t|^{1+\kappa}}$ converges since $\kappa' < \kappa$ but its value tends to $\infty $ as $\kappa' \to \kappa - 0$. 
Finally, we have
\begin{eqnarray*}
    \EE_\xi[M'_2(\xi, x, y)^{\kappa'}] &=& \EE_\xi\left[\left| M_2 + \frac{|\phi(\xi| x,y)|}{\|x - y\|_2}\right| ^{\kappa'}\right] \\
    &\overset{\text{Jensen inq, $\kappa' > 1$}}{\leq}& 2^{\kappa' - 1} \left[ M_2^{\kappa'}+ \frac{\EE_\xi\left[ |\phi(\xi| x,y)|^{\kappa'}\right]}{\|x - y\|_2^{\kappa'}}\right]\\
     &\leq&2^{\kappa' - 1} \left[ M_2^{\kappa'}+     I(\kappa') \Delta^{\kappa'}\right].
\end{eqnarray*}
Hence, $M'_2\!=\! (\EE_\xi[M'_2(\xi, x, y)^{\kappa'}])^\frac{1}{\kappa'}\! =\! O(M_2 + \Delta)$, where constant in $O(\cdot)$ depends only on $\kappa'$. 
\end{proof}

\begin{remark}[Role of the scale function $B(x,y)$]\label{rmk:role}
    In inequality \eqref{eq:dist_ass} due to normalization property of probability density, we must ensure that
    \begin{equation*}
        \int \limits_{-\infty}^{+\infty} \frac{\gamma^\kappa |B(x,y)|^{\kappa}}{|B(x,y)|^{1+\kappa} + |u|^{1+\kappa}} du \geq \int \limits_{-\infty}^{+\infty} p(u|x,y) du  = 1.
    \end{equation*}
    One can make substitution $t = u/|B(x,y)|$ and ensure that for $\kappa \leq 2$
    \begin{equation*}
        \int \limits_{-\infty}^{+\infty} \frac{\gamma^\kappa |B(x,y)|^{\kappa} du}{|B(x,y)|^{1+\kappa} + |u|^{1+\kappa}}  =  \gamma^\kappa\int \limits_{-\infty}^{+\infty} \frac{dt}{ 1 + |t|^{1+\kappa}} \overset{\kappa = 1}{\geq} \gamma^\kappa \pi .
    \end{equation*}
    Hence,  $\gamma$ is sufficient to satisfy
    \begin{eqnarray*}
        \gamma \geq \left(\frac{1}{\pi}\right)^\frac{1}{\kappa}.
    \end{eqnarray*}

    As scale value $|B(x,y)|$ decreases, quantiles of $p(u|x,y)$ gets closer to zero. Therefore, $|B(x,y)|$ can be considered as analog of variance of distribution $p(u|x,y)$.
\end{remark}

\begin{remark}[Standard oracles examples]\label{rem:examples}
    To build noise $\phi(\xi|x,y)$ satisfying Assumption~\ref{as:dist} with $\kappa > 0$ we will use independent random variables $\{\xi_k\}$ with symmetric probability density functions $p_{\xi_k}(u)$
    \begin{eqnarray*}
        p_{\xi_k}(u) \leq \frac{|\gamma_k \Delta_k|^\kappa}{|\Delta_k|^{1 + \kappa} + |u|^{1+\kappa}}, \quad \Delta_k, \gamma_k > 0,    
    \end{eqnarray*}
    such that for any real numbers $\{a_k\}_{k=1}^n$ and sum $\sum \limits_{k=1}^n a_k \xi_k$ it holds

    \begin{eqnarray}\label{eq:addition_law}
        p_{\sum \limits_{k=1}^n a_k\xi_k} (u) \leq \frac{\left(\sum \limits_{k=1}^n| \gamma_k a_k   \Delta_k|\right)^\kappa}{\left(\sum \limits_{k=1}^n|a_k \Delta_k|\right)^{1 + \kappa} + |u|^{1+\kappa}}.
    \end{eqnarray}
    Moreover, using Cauchy-Schwarz inequality  we bound 
    \begin{equation}\label{eq:Cauchy_sch}
        \sum \limits_{k=1}^n| \gamma_k a_k   \Delta_k| \leq  \|(\gamma_1\Delta_1, \dots, \gamma_n\Delta_n)^\top\|_2 \cdot \|(a_1, \dots, a_k)^\top\|_2.
    \end{equation}
    For example, variables $\xi_k$ can have Cauchy distribution with $\kappa = 1$ and $p(u) = \frac{1}{\pi}\frac{\Delta_k}{\Delta_k^2 + u^2}$ parametrized by scale $\Delta_k$. For the independent Cauchy variables with scales $\{\Delta_k\}_{k=1}^n$ and any real numbers $\{a_k\}_{k=1}^n$, the sum $\sum \limits_{k=1}^n a_k \xi_k$ is the Cauchy variable with scale $\sum \limits_{k=1}^n |a_k| \Delta_k$.  Therefore, inequality \eqref{eq:addition_law} for Cauchy variables holds true. For oracles, we have:

    \begin{itemize}
        \item \textbf{Independent oracle:}
        
        $f(x, \xi) = f(x) + \xi_x, f(y,  \xi) = f(y) + \xi_y, \phi(\xi| x,y) = \xi_x - \xi_y$, where $\xi_x, \xi_y$ are independent samples for each point $x$ and $y$. Thus, we have the final scale $\Delta = \Delta_x  + \Delta_y.$

        \item \textbf{Lipschitz oracle:}

        $f(x, \boldsymbol{\xi}) = f(x) + \langle \boldsymbol{\xi}, x \ra , f(y, \boldsymbol{\xi}) = f(y) + \langle \boldsymbol{\xi}, y \ra, \phi(\boldsymbol{\xi}| x, y) = \la \boldsymbol{\xi}, x - y \ra$, where $\boldsymbol{\xi}$ is $d$-dimensional random vector with components $\xi_k$. Oracle gives the same realization of $\boldsymbol{\xi}$  for both $x$ and $y$. In that case, the vector $\boldsymbol{\xi}$ can be restated to $\boldsymbol{\xi} = A\boldsymbol{\xi}_{ind}$ with $\phi(\boldsymbol{\xi}| x, y) = \la \boldsymbol{\xi}_{ind}, A^\top(x - y)\ra$, where $A$ is the correlation matrix and $\boldsymbol{\xi}_{ind}$ are independent Cauchy variables. Now, if the vector $\boldsymbol{\xi}_{ind}$ has scales $\{\Delta_k\}_{k=1}^n$, then we have $\gamma$ and $B(x,y)$ from Assumption~\ref{as:dist} equal to \begin{eqnarray*}
            \gamma &=& \frac{1}{\pi}, \\
            B(x,y) &=& \sum \limits_{k=1}^d| \Delta_k [A^\top(x-y)]_k| \overset{\eqref{eq:Cauchy_sch}}{\leq} \|(\Delta_1, \dots, \Delta_d)^\top\|_2 ||A^\top||_2||x-y||_2.
        \end{eqnarray*}
        
    \end{itemize}

\end{remark}

\section{Proofs}\label{sec:proofs}
\subsection{Proof of Median Estimate Properties Lemma~\ref{lem:batch median properties}.} \label{sec:lemma}
To begin with, we need some properties of the smoothed approximation  $\hat{f}_\tau$.

\begin{proposition}
[Strong convexity of $\hat{f}_\tau$] Consider $\mu$-strongly convex (As. \ref{as:f convex}) function $f$ on $Q \subseteq \R^d$. Then the smoothed function $\hat{f}_\tau$ defined in \eqref{hat_f} is also $\mu$-strongly convex on $Q$.
\end{proposition}

\begin{proof}
The function $f$ is $\mu$-strongly convex if we have $\forall x, y \in Q$ and $\forall t \in [0,1]$
$$f(x t + y (1-t)) \leq t \cdot f(x) + (1-t)  \cdot f(y) - \frac{1}{2}\mu t (1-t) \|x-y\|_2^2. $$
Following the definition of $\hat{f}_\tau$, we write down for $\mathbf{u} \in U(B_1(0))$ inequality
\begin{eqnarray*}
    f(xt + y (1-t) + \tau\mathbf{u} ) &=& f((x + \tau \mathbf{u}) \cdot t + (y + \tau \mathbf{u})  \cdot (1-t)) \\
    &\leq& t  \cdot f(x + \tau \mathbf{u}) + (1-t)  \cdot f(y + \tau \mathbf{u}) - \frac{1}{2}\mu t (1-t) \|x-y\|_2^2.
\end{eqnarray*}
Taking math expectation $\EE_{\mathbf{u}}$ from both sides, we have
$$\EE_{\mathbf{u}}[f(xt + y (1-t) + \tau\mathbf{u} )] \leq t  \cdot \EE_{\mathbf{u}}[f(x + \tau \mathbf{u})] + (1-t)  \cdot \EE_{\mathbf{u}}[f(y + \tau \mathbf{u})] - \frac{1}{2}\mu t (1-t) \|x-y\|_2^2. $$ 
\end{proof}

\begin{proof}[Proof of Lemma~\ref{lem:batch median properties}.]
First, we notice from our construction of the oracle that
$$f(x, \xi) - f(y, \xi) = f(x) - f(y)  + \phi(\xi| x,y), \quad \forall x,y \in Q, $$
and we have
\begin{eqnarray}
    g(x, \mathbf{e}, \xi) &=& \frac{d}{2\tau}(f(x + \tau \mathbf{e}, \xi) - f(x - \tau \mathbf{e}, \xi)) \notag \\
    &=& \frac{d}{2\tau} [f(x + \tau \mathbf{e}) - f(x - \tau \mathbf{e})]\mathbf{e} \notag + \frac{d}{2\tau }\phi(\xi|x + \tau \mathbf{e}, x - \tau \mathbf{e}) \mathbf{e} \notag
\end{eqnarray}  
and for $\texttt{Med}^m(x, \mathbf{e}, \{\xi\})$ we have
\begin{eqnarray}
     \texttt{Med}^m(x, \mathbf{e}, \{\xi\}) &=& \texttt{Median} \left( \left\{g(x, \mathbf{e}, \xi^i)\right\}_{i=1}^{2m+1}\right) \notag \\
    &=& \texttt{Median} \left( \left\{\frac{d}{2\tau} [f(x + \tau \mathbf{e}) - f(x - \tau \mathbf{e})]\mathbf{e}
    + \frac{d}{2\tau }\phi(\xi^i| x + \tau \mathbf{e}, x - \tau \mathbf{e})\mathbf{e}\right\}_{i=1}^{2m+1}\right) \notag \\
    &=&  \frac{d}{2\tau} [f(x + \tau \mathbf{e}) - f(x - \tau \mathbf{e})]\mathbf{e} \label{eq:med_prop_1_term}\\
    &+& \frac{d}{2\tau }\texttt{Median} \big(\big\{\phi(\xi^i| x + \tau \mathbf{e}, x - \tau \mathbf{e})\big\}_{i=1}^{2m+1} \big) \mathbf{e}.\label{eq:med_prop_2_term}
\end{eqnarray}
\noindent\textbf{Finite second moment:}
Further, we analyze two terms: gradient estimation term \eqref{eq:med_prop_1_term} and the noise term \eqref{eq:med_prop_2_term}. Following work \citep{kornilov2023gradient} [Lemma~$2.3$], we have the upper bound for the second moment of \eqref{eq:med_prop_1_term}
\begin{eqnarray}\label{eq:med_prop_1_upper}
    \EE_{\mathbf{e}} \left[ \left|\left|\frac{d}{2\tau} [f(x + \tau \mathbf{e}) - f(x - \tau \mathbf{e})]\mathbf{e}\right|\right|_q^2\right] \leq da_q^2M_2^2, 
\end{eqnarray}
where $a_q = d^{\frac1q - \frac12} \min\{\sqrt{32\ln d - 8}, \sqrt{2q - 1}\}$ is a special coefficient, such that,
\begin{equation}\label{eq:med_prop_a_q}
    \EE_{\mathbf{e}}[\|\mathbf{e}\|_q^2] \leq a_q^2.
\end{equation}
See Lemma 2.1 from \citep{gorbunov2022accelerated} and Lemma $8.4$ from \citep{kornilov2023gradient} for more details.

Next, we deal with noise term \eqref{eq:med_prop_2_term}. For symmetric variable  $\phi(\xi| x,y)$ for all $x,y \in Q$ under Assumption~\ref{as:dist} it holds
    \begin{eqnarray}
         p(u) \leq \frac{\gamma^\kappa |B(x,y)|^\kappa}{ |B(x,y)|^{1+\kappa} + |u|^{1+\kappa}}. \notag 
    \end{eqnarray}
Further, we prove that, for large enough $m$, noise term has finite variance. For this purpose, we denote $Y \eqdef \texttt{Median} \big(\big\{\phi(\xi^i|x,y)\big\}_{i=1}^{2m+1} \big)$ and cumulative distribution function of $Y$ 
    $$P(t) \eqdef \int\limits_{-\infty}^t p(u)du.$$
    Median of $2m+1$ i.i.d. variables distributed according to  $p(u)$ is $(m+1)$-th order statistic, which has probability density function 
    $$(2m+1)\begin{pmatrix}
        2m \\ m
    \end{pmatrix} P(t)^m (1 - P(t))^m p(t).$$
The second moment $\EE[Y^2]$  can be calculated via
    \begin{eqnarray*}
        \EE[Y^2] &=&\int \limits_{-\infty}^{+\infty} (2m+1)\begin{pmatrix}
        2m \\ m
    \end{pmatrix} t^2 P(t)^m(1 -  P(t))^m p(t) dt \\
        &\leq& (2m+1)\begin{pmatrix}
        2m \\ m
    \end{pmatrix} \sup_t\{t^2 P(t)^m(1 -  P(t))^m\}   \int \limits_{-\infty}^{+\infty} p(t) dt \\
        &\leq& (2m+1)\begin{pmatrix}
        2m \\ m
    \end{pmatrix} \sup_t\{t^2 P(t)^m(1 -  P(t))^m\} .
    \end{eqnarray*}
    For any $t < 0$, we have 
    \begin{eqnarray*}
        P(t) &=& \int\limits_{-\infty}^t p(u)du \leq  \int\limits_{-\infty}^t \frac{|\gamma B(x,y)|^\kappa}{|B(x,y)|^{1+\kappa} + |u|^{1+\kappa}} \notag \\
        &\leq&  \int\limits_{-\infty}^t \frac{|\gamma B(x,y)|^\kappa}{|u|^{1+\kappa}} \leq \frac{| \gamma B(x,y)|^\kappa}{\kappa } \cdot \frac{1}{|t|^\kappa}.
    \end{eqnarray*}   
    Similarly, one can prove that for any $t > 0$
    $$1 -  P(t) = \int\limits^{\infty}_t p(u)du \leq \frac{|\gamma B(x,y)|^\kappa}{\kappa } \cdot \frac{1}{t^\kappa}.$$
    Since for any number $a \in [0,1]$ holds $a(1 - a) \leq \frac14$ we have for any $t \in \R$
    $$P(t)(1 -  P(t)) \leq \min\left\{\frac14, \frac{|\gamma B(x,y)|^\kappa}{\kappa } \cdot \frac{1}{|t|^\kappa}\right\}$$
    along with
    \begin{eqnarray}
         t^2 P(t)^m(1 -  P(t))^m \leq \min\left\{\frac{t^2}{4^m}, \left(\frac{|\gamma B(x,y)|^\kappa}{\kappa }\right)^m \cdot \frac{1}{|t|^{m\kappa - 2}}\right\}.\label{eq:med_prop_under_int}
    \end{eqnarray}
    If $m\kappa > 2$ the first term of \eqref{eq:med_prop_under_int} increasing and the second one decreasing with the growth of $|t|$, then the maximum of the minimum \eqref{eq:med_prop_under_int} is achieved when 
    \begin{eqnarray*}
        \frac{t^2}{4^m} &=& \left(\frac{|\gamma B(x,y)|^\kappa}{\kappa }\right)^m \cdot \frac{1}{|t|^{m\kappa - 2}}, \\
        |t| &=& |\gamma B(x,y)|\left( \frac{4}{\kappa }\right)^\frac1\kappa.
    \end{eqnarray*}
    Therefore, we get for any $t \in \R$
    $$ t^2 P(t)^m(1 -  P(t))^m \leq \frac{|\gamma B(x,y)|^2}{4^m} \left( \frac{4}{\kappa }\right)^\frac2\kappa,$$
    and, as a consequence
    \begin{eqnarray*}   
    \EE[Y^2] \leq  (2m+1)\begin{pmatrix}
        2m \\ m
    \end{pmatrix}\frac{|\gamma B(x,y)|^2}{4^m} \left( \frac{4}{\kappa }\right)^\frac2\kappa.
    \end{eqnarray*} 
It only remains to note
    \begin{eqnarray*}
        \begin{pmatrix}
        2m \\ m
    \end{pmatrix} = \frac{(2m)!}{m! \cdot m!} = \prod \limits_{j=1}^m \frac{2j}{j} \cdot \prod \limits_{j=1}^m \frac{2j - 1}{j} \leq 4^m. 
    \end{eqnarray*}
Since $Y$ has the finite second moment, it has finite math expectation
$$\EE[Y] = \int \limits_{-\infty}^{+\infty} (2m+1)\begin{pmatrix}
        2m \\ m
    \end{pmatrix} t P(t)^m(1 -  P(t))^m p(t) dt.$$
    For any $t \in \R$, due to symmetry of $p(t)$, we have $P(t) = (1 - P(-t))$ and $p(t) = p(-t)$ and, as a consequence,
$$\EE[Y] = \int \limits_{-\infty}^{+\infty} (2m+1)\begin{pmatrix}
        2m \\ m
    \end{pmatrix} t P(t)^m(1 -  P(t))^m p(t) dt = 0.$$
    Finally, we have an upper bound for \eqref{eq:med_prop_2_term}
    \begin{eqnarray}
        \EE_{\mathbf{e}, \xi}\left|\left|\frac{d}{2\tau }\texttt{Median} \big(\big\{\phi(\xi^i |x + \tau \mathbf{e},x - \tau \mathbf{e})\big\} \big) \mathbf{e} \right|\right|_q^2 
        &=& \left(\frac{d}{2\tau}\right)^2 \EE_{\mathbf{e}}[ \EE_\xi[Y^2| \mathbf{e}] \cdot \|\mathbf{e}\|_q^2] \notag\\
        &\leq&   \left(\frac{d}{2\tau}\right)^2  (2m+1) \left( \frac{4}{\kappa }\right)^\frac2\kappa \cdot  \EE_\mathbf{e}[|\gamma B(x + \tau \mathbf{e}, x - \tau \mathbf{e} )|^2\|\mathbf{e}\|_q^2]. \notag\label{eq:med_prop_exp} 
    \end{eqnarray}
In case of the \textbf{independent} oracle, we simplify using Assumption~\ref{as:dist} and \eqref{eq:Tsubakov}
\begin{eqnarray}
    \EE_\mathbf{e}[|\gamma B(x + \tau \mathbf{e}, x - \tau \mathbf{e} )|\|\mathbf{e}\|_q^2] \leq \Delta^2 \EE_\mathbf{e}[\|\mathbf{e}\|_q^2] \overset{\eqref{eq:med_prop_a_q}}{\leq} \Delta^2 a_q^2. \label{eq:med_prop_one_bound}
\end{eqnarray}
In case of the \textbf{Lipschitz} oracle, we use  \eqref{eq:noise_lip} and get
    \begin{eqnarray}
    \EE_\mathbf{e}[|\gamma B(x + \tau \mathbf{e}, x - \tau \mathbf{e} )|\|\mathbf{e}\|_q^2] &\leq& 4\Delta^2\tau^2  \EE_\mathbf{e}[\|\mathbf{e}\|_2^2\|\mathbf{e}\|_q^2] \overset{\eqref{eq:med_prop_a_q}}{\leq} 4\Delta^2\tau^2 a_q^2. \notag \label{eq:med_prop_lip_bound}
\end{eqnarray}
    Combining upper bounds \eqref{eq:med_prop_1_upper}
and \eqref{eq:med_prop_one_bound} or \eqref{eq:med_prop_lip_bound}, we obtain total bound
\begin{eqnarray*}
    \EE_{\mathbf{e}, \xi}[\|\texttt{Med}^m(x, \mathbf{e}, \{\xi\})\|_q^2] &\leq& 2 \cdot \eqref{eq:med_prop_1_upper}  + 2 \cdot \eqref{eq:med_prop_one_bound}\\\eqref{eq:med_prop_lip_bound}.
\end{eqnarray*}
    For the batched gradient estimation $\texttt{BatchMed}^{m}_b(x, \{\mathbf{e}\}, \{\xi\})$ and $q = 2$, we use Lemma~$4$ from \citep{kornilov2023accelerated} that states
    \begin{eqnarray*}
        &&\EE_{\mathbf{e}, \xi}[\|\texttt{BatchMed}^{m}_b(x, \{\mathbf{e}\}, \{\xi\})\|_2^2] \leq \frac{1}{b}  \cdot \EE_{\mathbf{e}, \xi}[\|\texttt{Med}^m(x, \mathbf{e}, \{\xi\})\|_2^2].
    \end{eqnarray*}
    For the bound of the centered second moment, we use Jensen's inequality for any random vector $X$  
    $$\EE[||X - \EE[X]||_q^2] \leq 2\EE[||X||_q^2] + 2||\EE[X]||_q^2 \leq 4\EE[||X||_q^2].$$
\textbf{Unbiasedness:} According to Lemma~\ref{lem:hat_f properties}, the term \eqref{eq:med_prop_1_term} is an unbiased estimation of the gradient $\nabla \hat{f}_\tau(x)$. Indeed, the distribution of $\mathbf{e}$ is symmetrical and we can derive
$$\EE_\mathbf{e}\left[ \frac{d}{2\tau} [f(x + \tau \mathbf{e}) - f(x - \tau \mathbf{e})]\mathbf{e}\right]  = \EE_\mathbf{e}\left[ \frac{d}{\tau} [f(x + \tau \mathbf{e})]\right]  = \nabla \hat{f}^\tau(x). $$
Since $Y$ has the finite second moment, it has finite math expectation
$$\EE[Y] = \int \limits_{-\infty}^{+\infty} (2m+1)\begin{pmatrix}
        2m \\ m
    \end{pmatrix} t P(t)^m(1 -  P(t))^m p(t) dt.$$
    For any $t \in \R$, due to symmetry of $p(t)$, we have $P(t) = (1 - P(-t))$ and $p(t) = p(-t)$ and, as a consequence,
$$\EE[Y] = \int \limits_{-\infty}^{+\infty} (2m+1)\begin{pmatrix}
        2m \\ m
    \end{pmatrix} t P(t)^m(1 -  P(t))^m p(t) dt = 0.$$
Hence, we obtained that $\EE_{\mathbf{e}, \xi} [\texttt{Med}^m(x, \mathbf{e}, \{\xi\})] = \nabla \hat{f}_\tau(x) $ along with $\EE_{\mathbf{e}, \xi} [\texttt{BatchMed}^{m}_b(x, \{\mathbf{e}\}, \{\xi\})] = \nabla \hat{f}_\tau(x)$ as the batching is the mean of random vectors with the same math expectation. 
\end{proof}

\subsection{Proof of \algname{ZO-clipped-med-SSTM} Convergence Theorem \ref{thm:SSTM col}} \label{par: sstm conv thms}

\begin{proof} For any point $x \in Q = B_{3R + 2\tau}{(x^*)}$, we can consider median estimations $\texttt{Med}^{m}(x, \mathbf{e}, \{\xi\})$ and  $\texttt{BatchMed}^{m}_b(x, \{\mathbf{e}\}, \{\xi\})$ to be an oracle for the gradient of $\hat{f}_\tau(x)$ that satisfies Assumption~\ref{as:oracle}.
\begin{assumption}\label{as:oracle}
    Let $G(x, \mathbf{e}, \xi)$ be the oracle for the gradient of function  $\hat{f}_\tau(x)$, such that for any point $x \in Q$ it is unbiased, i.e.,
    \begin{eqnarray}
        \EE_{\mathbf{e}, \xi} [G(x, \mathbf{e}, \xi)] = \nabla \hat{f}_\tau(x), \notag
    \end{eqnarray}
    and has bounded second moment, i.e.,
    \begin{eqnarray}\label{eq:oracle_bound main}
        \EE_{\mathbf{e}, \xi} [\|G(x, \mathbf{e}, \xi) - \nabla \hat{f}_\tau(x)\|^2_q] \leq \Sigma^2_q,
    \end{eqnarray}
    where $\Sigma_q$ might depend on $\tau$.
\end{assumption}
Thus, in order to prove convergence of our \algname{ZO-clipped-med-SSTM}, we use the general convergence theorem with oracle satisfying Assumption~\ref{as:oracle} for \algname{ZO-clipped-SSTM} (Theorem~$1$ from  \citep{kornilov2023accelerated} with $\alpha = 2$).  Next, we take $\texttt{BatchMed}^{m}_b(x, \{\mathbf{e}\}, \{\xi\})$ as the necessary oracle and substitute $\Sigma_2$ from \eqref{eq:oracle_bound main} with $\sigma/\sqrt{b}$ from Lemma~\ref{lem:batch median properties}.

\begin{theorem}[Convergence of \algname{ZO-clipped-SSTM}] We denote  $R = \|x^0 - x^*\|_2$, where $x^0$
is a starting point and $x^*$ is an optimal solution to \eqref{eq:min_problem}. Consider convex (As. \ref{as:f convex}) and $M_2$-Lipschitz (As. \ref{as:Lipshcitz}) function $f$ on $B_{3R}(x^*)$ with gradient oracle under As.~\ref{as:oracle} with $\Sigma_2$. 

We run \algname{ZO-clipped-SSTM} for $K$ iterations with smoothing parameter $\tau$, batch size $b$,  probability $1-\beta$ and further parameters $  A = \ln \nicefrac{4K}{\beta}\geq 1$,
$a = \Theta(\min\{A^2, \nicefrac{\Sigma_2 K^{2}\sqrt{A}\tau}{\sqrt{d b}M_2R}\}), \lambda_k = \Theta(\nicefrac{R}{(\alpha_{k+1}A)})$.
We guarantee that with probability at least $1-\beta$:
\begin{equation}
    f(y^K) - f(x^*) = 2M_2 \tau +  \widetilde{\cO}\left( \max\left\{\frac{\sqrt{d} M_2 R^2}{\tau K^2}, \frac{\Sigma_2 R}{\sqrt{bK }}\right\}\right). \label{eq:med_stmm conv before tau main}
\end{equation}
Moreover, with probability at least  $ 1-\beta$ the iterates of \algname{ZO-clipped-SSTM} remain in the ball with center $x^*$ and radius $2R$, i.e., $\{x^k\}_{k=0}^{K+1}, \{y^k\}_{k=0}^{K}, \{z^k\}_{k=0}^{K} \subseteq B_{2R}(x^*)$.
\end{theorem}
The statement of Theorem \ref{thm:SSTM col} follows if we put $\Sigma_2 = \sigma/\sqrt{b}$ and equate both terms of \eqref{eq:med_stmm conv before tau main} to $\frac{\varepsilon}{2}$,  taking $\tau = \frac{\varepsilon}{4 M_2}$. 
\end{proof}

\subsection{Proof of \algname{ZO-clipped-med-SMD} Convergence Theorem \ref{thm:SMD col}} \label{par: conv thms}

In order to prove convergence of  \algname{ZO-clipped-med-SMD}, we use the general convergence theorems with oracle satisfying Assumption~\ref{as:oracle} for \algname{ZO-clipped-SMD} (Theorem $4.3$ from \citep{kornilov2023gradient}  with $\kappa = 1$).  Next, we take  $\texttt{Med}^{m}(x, \mathbf{e}, \{\xi\})$ as the necessary oracle and substitute $\Sigma_q$ from \eqref{eq:oracle_bound main} with $\sigma a_q $ from Lemma~\ref{lem:batch median properties}.

\begin{theorem}[Convergence of \algname{ZO-clipped-SMD}]
 Consider convex (As. \ref{as:f convex}) and $M_2$-Lipschitz (As. \ref{as:Lipshcitz}) function $f$ on a convex compact $Q$ with gradient oracle under As.~\ref{as:oracle} with $\Sigma_q$. 

We run \algname{ZO-clipped-SMD} for $K$ iterations with smoothing parameter $\tau$, norm $q \in [2, +\infty]$, prox-function $\Psi_p$, probability $1-\beta$ and further parameters $\lambda = \Sigma_q \sqrt{K}$, $\nu = \frac{D_{\Psi_p}}{\lambda}$, where squared diameter $D_{\Psi_p}^2 \eqdef 2 \sup\limits_{x,y \in Q}  V_{\Psi_{p}}(x,y).$
We guarantee that with probability at least $1-\beta$:
\begin{equation}
    f(\overline{x}^K) - f(x^*) = 2M_2 \tau +  \widetilde{\cO}\left(  \frac{\Sigma_q D_{\Psi_{p}}}{\sqrt{ K }} \right). \label{eq:med_smd conv before tau}
\end{equation}
\end{theorem}
The statement of Theorem \ref{thm:SMD col} follows if we equate both terms of \eqref{eq:med_smd conv before tau} to $\frac{\varepsilon}{2}$,  taking $\tau = \frac{\varepsilon}{4 M_2}$ and explicit formulas for $\sigma$ and $a_q$ from Lemma~\ref{lem:batch median properties}.

\noindent \textbf{Explicit parameters for the standard convex compacts.} In this paragraph, we discuss some standard sets $Q$ and prox-functions $\Psi_p$ taken from \citep{ben2001lectures}. We can choose prox-functions to reduce $a_q D_{\Psi_{p}}$ and get better convergence constants. The two main setups are 
\begin{enumerate}
    \item Ball setup, $p = 2, q = 2$:
\begin{equation*}\label{eq: ball setup} \qquad  \Psi_p(x) = \frac{1}{2}\|x\|_2^2,
    \end{equation*}
    \item Entropy setup, $p = 1, q = \infty$:
\begin{equation*}  \qquad  \Psi_p(x) = (1+\gamma)\sum_{i=1}^d (x_i + \gamma/d)\log(x_i + \gamma/d).
    \end{equation*}
\end{enumerate}
We consider unit balls and standard simplex $\triangle^d_+$ as $Q$. For $Q = \triangle^d_+$ or unit $\ell_1$-ball, the Entropy setup is preferable. Meanwhile, for unit $\ell_2$-ball or  $\ell_\infty$-ball, the Ball setup is better.

\subsection{Proof of \algname{Clipped-INF-med-SMD} Convergence Theorem~\ref{avrregretest}}
\label{app:proof_of_th_3}
We start the proof with the following several lemmas.
\begin{lemma}
    \label{ftau}
    Let $f(x)$ be a linear function, then 
    $\nabla f(x) = \nabla \hat{f}_{\tau}(x)$.
\end{lemma}
\begin{proof}
\begin{align*}
    \nabla \hat{f}_{\tau}(x) &=  \nabla \EE_{\mathbf{u} \sim B_1(0)} [f(x + \tau \mathbf{u})] = \nabla \EE_{\mathbf{u}\sim B_1(0)} [\langle \mu, x + \tau \mathbf{u} \rangle] \\
    &= \nabla \langle \mu, x+\tau \EE_{\mathbf{u} \sim B_1(0) }[u] \rangle = \nabla \langle \mu, x \rangle = \nabla f(x). 
\end{align*}
\end{proof}
\begin{lemma}
    \label{medbound}
    Let $f(x)$ be a linear function, $q = \infty$, $\tau = \alpha \sqrt{d}$, then 
    $$
\EE_{\mathbf{e}, \xi}[\|g_{med}^{k+1} - \mu\|_{\infty}^2] \leq \left ( 32 \ln d - 8 \right ) \cdot \left (8M_2^2 +  2\alpha^2 \Delta^2 (2m+1)\left(\frac{4}{\kappa} \right)^{\frac2\kappa} \right ).
$$
\end{lemma}
\begin{proof}
From \ref{lem:batch median properties} with $q = \infty$ and $\tau = \alpha \sqrt{d}$ we get
$$\EE_{\mathbf{e}, \xi}[\|\texttt{Med}^{m}(x, \mathbf{e}, \{\xi\}) -\nabla \hat{f}_\tau(x)\|_{\infty}^2] \leq   \sigma^2a_{\infty}^2,\quad a_{\infty} = d^{- \frac12} \sqrt{32\ln d - 8},$$ 
where $ \sigma^2 = d  \left (8M_2^2 +  2\alpha^2 \Delta^2 (2m+1)\left(\frac{4}{\kappa} \right)^{\frac2\kappa} \right )$.
Hence, w.r.t \eqref{ftau} we get
$$
\EE_{\mathbf{e}, \xi}[\|g_{med}^{k+1} - \mu\|_{\infty}^2] \leq \left ( 32 \ln d - 8 \right ) \cdot \left (8M_2^2 +  2\alpha^2 \Delta^2 (2m+1)\left(\frac{4}{\kappa} \right)^{\frac2\kappa} \right ).
$$
\end{proof}
\begin{lemma}
\label{sad}
    \textbf{[Lemma 5.1 from \citep{sadiev2023high}]} Let $X$ be a random vector in $\R^d$ and $\bar{X} = \clip(X, \lambda)$, then 
\begin{equation}
     \|\bar{X} - \EE[\bar{X}]\| \leq 2\lambda.
\end{equation}
Moreover, if for some $c \geq 0$ 
\begin{equation*}
    \EE[X] = x \in \R^n, \quad \EE[\|X - x\|^{2}] \leq c^{2}
\end{equation*}
and $\|x\| \leq \frac{\lambda}{2}$, then 
\begin{align}
    \left \| \EE[\bar{X}] - x \right \| &\leq \frac{4 c^{2}}{\lambda},\\
    \EE\left [ \left \| \bar{X} - x   \right \|^2 \right] &\leq 18 c^{2},\\
    \EE\left [ \left \| \bar{X} - \EE[\bar{X}]   \right \|^2 \right] &\leq 18 c^{2}.
\end{align}
\end{lemma}
\begin{remark}\label{useful bounds}
    Combination of Lemma \ref{medbound} and Lemma \ref{sad} with $X = g_{med}^{k(t)}$ and $x = \mu$ in case when $\lambda \geq 2\|\mu\|_{\infty}$ immidiatly get the following bounds:
    $$
    \left \| \EE [g_{med}^{k(t)}] - \EE[\tilde{g}_{med}^{k(t)}] \right\|_{\infty} =  \left \| \mu - \EE[\tilde{g}_{med}^{k(t)}] \right\|_{\infty} \leq \frac{4 c^2}{\lambda},
    $$
    $$
    \EE \left [ \| \tilde{g}_{med}^{k(t)} \|_{\infty}^2 \right ] \leq 2 \EE \left [  \| \tilde{g}_{med}^{k(t)} - \mu \|_{\infty}^2 
 + \|\mu\|_{\infty}^2\right ] \leq 2\|\mu\|_{\infty}^2 + 36 c^2, 
    $$
    for $c^2 = \left ( 32 \ln d - 8 \right ) \cdot \left (8M_2^2 +  2\alpha^2 \Delta^2 (2m+1)\left(\frac{4}{\kappa} \right)^{\frac2\kappa} \right ).$
\end{remark}

\begin{lemma}
\label{regretest1}
    Suppose that \algname{Clipped-INF-med-SMD} with $1/2$-Tsallis entropy $$\psi(x) = 2\left (1 - \sum_{i=1}^d x_i^{1/2} \right), \quad x\in \Delta_+^d$$  as prox-function generates the sequences $\{x_k\}_{k=0}^K$ and $\{\tilde{g}_{med}^{k}\}_{k=0}^K$, then for any $u \in \Delta_+^d$ holds:
\begin{align*}
    &\sum_{k=0}^K \sum_{s=1}^{2m+1} \langle \tilde{g}_{med}^{k}, x_{k} - u \rangle  \leq (2m+1) \left [2\frac{d^{1/2} - \sum_{i=1}^d u_i^{1/2}}{\nu} + \nu\sum_{k=0}^K \sum_{i=1}^d  (\langle \tilde{g}_{med}^{k})_i^2 \cdot x_{k,i}^{3/2} \right ].
\end{align*}
\end{lemma}
\begin{proof}
By definition, the \textit{Bregman divergence} $V_{\psi} (x, y)$ is:
\begin{eqnarray}
V_{\psi} (x, y) &=& \psi(x) - \psi(y) -\langle\nabla\psi(y), x-y\rangle \notag \\ 
&=& 2\left ( 1 - \sum_{i=1}^d x_i^{1/2} \right ) - 2 \left ( 1 - \sum_{i=1}^d y_i^{1/2} \right ) + \sum_{i=1}^d y_i^{-1/2}(x_i-y_i) \notag \\
&=& -2\sum_{i=1}^d x_i^{1/2}+2\sum_{i=1}^d y_i^{1/2} + \sum_{i=1}^d y_i^{-1/2}(x_i-y_i). \notag
\end{eqnarray}
Note that the algorithm can be considered as an online mirror descent (OMD) with batching and the Tsallis entropy used as prox-function: 
\begin{align*}
    &x_{k+1} = \arg \min_{x \in \Delta_+^d} \left [ \nu x^{\mathtt{T}} \tilde{g}_{med}^{k} + V_{\psi}(x, x_k)   \right ].
\end{align*}
Thus, the standard inequality for OMD holds:
\begin{align}\label{OMD1}
    \langle \tilde{g}_{med}^{k}, x_k - u \rangle \leq  \frac{1}{\nu} \left [ V_{\psi}(u, x_k) - V_{\psi}(u, x_{k+1}) - V_{\psi}(x_{k+1}, x_k) \right]  +  \langle \tilde{g}_{med}^{k}, x_k - x_{k+1} \rangle.
\end{align}
From Tailor Theorem, we have 
\begin{align*}
    V_{\psi}(z, x_k) = \frac{1}{2} (z - x_k)^T \nabla^2 \psi(y_k) (z -  x_k) = \frac{1}{2}\|z - x_k\|^2_{\nabla^2\psi(y_k)}
\end{align*}
for some point $y_k \in [z, x_k]$. Hence, we have
\begin{align*}
    \langle \tilde{g}_{med}^{k}, x_k - x_{k+1} \rangle - \frac{1}{\nu}V_{\psi}(x_{k+1}, x_k) &\leq \max_{z \in R^d_+} \left [\langle \tilde{g}_{med}^{k}, x_k - z \rangle - \frac{1}{\nu}V_{\psi}(z, x_k) \right ]\\
    &=\left [\langle \tilde{g}_{med}^{k}, x_k - z^*_k \rangle - \frac{1}{\nu}V_{\psi}(z^*_k, x_k) \right ] \\
    &\leq \frac{\nu}{2}\|\tilde{g}_{med}^{k}\|^2_{(\nabla^2 \psi (y_k))^{-1}}+ \frac{1}{2}\|z^* - x_k\|^2_{\nabla^2\psi(y_k)} - \frac{1}{\nu}V_{\psi}(z^*, x_k) \\
    &=\frac{\nu}{2}\|\tilde{g}_{med}^{k}\|^2_{(\nabla^2 \psi (y_k))^{-1}},
\end{align*}
where $z^* = \arg \max_{z \in \R^d_+} \left [\langle \tilde{g}_{med}^{k}, x_k - z \rangle - \frac{1}{\nu}V_{\psi}(z, x_k) \right ]$. Proceeding with (\ref{OMD1}), we get:
\begin{align*}
    &\langle \tilde{g}_{med}^{k}, x_k - u \rangle \leq \frac{1}{\nu} \left [ V_{\psi}(u, x_k) - V_{\psi}(u, x_{k+1})  \right] + \frac{\nu}{2}\|\tilde{g}_{med}^{k}\|^2_{(\nabla^2 \psi (y_k))^{-1}}.
\end{align*}
Sum over $k$ gives
\begin{eqnarray}
    \sum_{k=0}^K \langle \tilde{g}_{med}^{k}, x_k - u \rangle    
    &\leq& \frac{V_{\psi} (x_0, u)}{\nu} + \frac{\nu}{2}\sum_{k=0}^K (\tilde{g}_{med}^{k})^T \left ( \nabla^2 \psi(y_k) \right )^{-1}\tilde{g}_{med}^{k} \notag \\
    &=&2\frac{d^{1/2} - \sum_{i=1}^d u_i^{1/2}}{\nu} + \nu\sum_{k=0}^K \sum_{i=1}^d  (\tilde{g}_{med}^{k})_i^2y_{k,i}^{3/2}, \label{sbound1}
\end{eqnarray}
where $y_k \in [x_k, z_k^*]$ and $z_k^* = \arg \max_{z \in R^d_+} \left [\langle \tilde{g}_{med}^{k}, x_k - z \rangle - \frac{1}{\nu}V_{\psi}(z, x_k) \right ]$. From the first-order optimality condition for $z^*_k$ we obtain
$$
-\nu (\tilde{g}_{med}^{k})_i + (x_{k, i})^{1/2} = (z_{k, i}^*)^{1/2}
$$
and thus we get $z_{k, i}^* \leq x_{k, i}$.  Finally, (\ref{sbound1}) becomes 
\begin{equation*}
    \sum_{k=0}^K \langle \tilde{g}_{med}^{k}, x_k - u \rangle   \leq 2\frac{d^{1/2} - \sum_{i=1}^d u_i^{1/2}}{\nu} + \nu\sum_{k=0}^K \sum_{i=1}^d  (\tilde{g}_{med}^{k})_i^2\cdot x_{k,i}^{3/2}
\end{equation*}
and concludes the proof. 
\end{proof}

\begin{lemma}
    \label{regretest2}
Suppose that \algname{Clipped-INF-med-SMD} with $1/2$-Tsallis entropy as prox-function generates the sequences $\{x_k\}_{k=0}^K$ and $\{\tilde{g}_{med}^{k}\}_{k=0}^K$, and for each arm $i$ random reward $g_{t,i}$ at any step $t$ has bounded expectation $\EE [g_{t, i}] \leq \frac{\lambda}{2}$ and the noise $g_{t,i} - \mu_i$ has symmetric distribution, then for any $u \in \Delta_+^d$ holds:
\begin{equation}
    \EE_{x_k, \mathbf{e}_{[k]}, \xi_{[k]}} \left [\sum_{i=1}^d  (\tilde{g}_{med}^{k})_i^2\cdot x_{k,i}^{3/2} \right ] \leq \sqrt{d} \cdot (2\|\mu\|_{\infty}^2 + 36c^2).
\end{equation}
\end{lemma}
\begin{proof}
\begin{align*}
    \EE_{x_k, \mathbf{e}_{[k]}, \xi_{[k]}} \left [\sum_{i=1}^d  (\tilde{g}_{med}^{k})_i^2\cdot x_{k,i}^{3/2} \right ] &\leq \EE_{x_k, \mathbf{e}_{[k]}, \xi_{[k]}} \left [\sum_{i=1}^d  (\tilde{g}_{med}^{k})_i^2\cdot x_{k,i}^{1/2} \right ] \\
    &\leq \EE_{x_k, \mathbf{e}_{[k]}, \xi_{[k]}} \left [\sqrt{\sum_{i=1}^d  (\tilde{g}_{med}^{k})_i^2} \cdot \sqrt{\tilde{g}_{med}^{k})_i^2\cdot x_{k,i}^{1/2}} \right ]\\
    &\leq \sqrt{ \EE_{x_k, \mathbf{e}_{[k]}, \xi_{[k]}} \left [\sum_{i=1}^d  (\tilde{g}_{med}^{k})_i^2 \right ]} \cdot \sqrt{ \EE_{x_k, \mathbf{e}_{[k]}, \xi_{[k]}} \left [(\tilde{g}_{med}^{k})_i^2\cdot x_{k,i}^{1/2}\right]}\\
    &\leq \sqrt{d} \cdot (2\|\mu\|_{\infty}^2 + 36c^2). 
\end{align*}
\end{proof}



\textbf{Theorem 3}
Consider MAB problem where the conditional probability density function for each loss satisfies Assumption~\ref{as:dist} with $\Delta, \kappa > 0,$ and $ \|\mu\|_{\infty} \leq R$. Then, for the period $T$,  the sequence $\{x_t\}_{t=1}^T$ generated  by \algname{Clipped-INF-med-SMD} with parameters $m = \frac{2}{\kappa}+1$, $\tau = \alpha \sqrt{d}$, $\nu = \frac{\sqrt{(2m+1)}}{\sqrt{T (36 c^2 + 2R^2)}}$, $\lambda = \sqrt{T}$ and prox-function $\psi(x) = 2\left (1 - \sum_{i=1}^d x_i^{1/2} \right)$  satisfies
\begin{equation}
    \EE \left [\mathcal{R}_T (u) \right ]  \leq \sqrt{Td} \cdot (8c^2/\sqrt{d} + 4 \sqrt{ (2m+1)(18 c^2 + R^2)}), \quad u \in \Delta_+^d,
\end{equation}
where $c^2 = \left ( 32 \ln d - 8 \right ) \cdot \left (8M_2^2 +  2\alpha^2 \Delta^2 (2m+1)\left(\frac{4}{\kappa} \right)^{\frac2\kappa} \right )$.  Moreover, high probability bounds from Theorem \ref{thm:SMD col} also hold.
\begin{proof}[Proof of Theorem~\ref{avrregretest}:]
First, for any $x,y \in \triangle^d_+$ we have 
\begin{equation}\label{eq:simlex_bound}
    \|x -y \|_2 \leq \sqrt{2}.
\end{equation}
Next, we obtain 
\begin{align*}
    \EE\left[\mathcal{R}_T(u)\right] &= \EE\left [\sum_{t=1}^T l(x_t) - \sum_{t=1}^T l(u) \right ] \leq \EE\left [ \sum_{t=1}^T \langle \nabla l(x_t), x_t - u \rangle  \right]\\
    &\leq \EE\left [ \sum_{t=1}^T \langle\mu - g_{med}^{k(t)}, x_{k(t)} - u \rangle  \right] +  \EE\left [ \sum_{t=1}^T \langle g_{med}^{k(t)} - \tilde{g}_{med}^{k(t)}, x_{k(t)}- u \rangle  \right]+ \EE\left [ \sum_{t=1}^T \langle \tilde{g}_{med}^{k(t)}, x_{k(t)} - u \rangle  \right]\\
    &=\EE\left [ \sum_{t=1}^T \langle g_{med}^{k(t)} - \tilde{g}_{med}^{k(t)} , x_{k(t)} - u \rangle  \right]+ \EE\left [ \sum_{t=1}^T \langle \tilde{g}_{med}^{k(t)}, x_{k(t)} - u \rangle  \right]\\
    &\leq \left [ \sum_{t=1}^T \| \EE [g_{med}^{k(t)}] - \EE[\tilde{g}_{med}^{k(t)}]\|_{\infty} \cdot \| x_{k(t)} - u\|_1  \right]+ \EE\left [ \sum_{t=1}^T \langle \tilde{g}_{med}^{k(t)}, x_{k(t)} - u \rangle  \right]\\
    & \overbrace{ \leq}^{\text{Remark \ref{useful bounds}, \eqref{eq:simlex_bound}}} \frac{8 c^2 T}{\lambda} + (2m+1)\EE\left [ \sum_{k=0}^K \langle \tilde{g}_{med}^k, x_k - u \rangle  \right]\\
    &\overbrace{ \leq}^{\text{Lemma \ref{regretest1}}} \frac{8 c^2 T}{\lambda} + (2m+1) \left [2\frac{d^{1/2} - \sum_{i=1}^d u_i^{1/2}}{\nu} + \nu\sum_{k=0}^K \sum_{i=1}^d  ( \tilde{g}_{med}^{k})_i^2 \cdot x_{k,i}^{3/2} \right ]\\
    &\overbrace{ \leq}^{\text{Lemma \ref{regretest2}}}  \frac{8 c^2 T}{\lambda}+ 2(2m+1)\frac{\sqrt{d}}{\nu} + \nu T \sqrt{d} (36 c^2 + 2\|\mu\|_{\infty}^2)\\
    &= \sqrt{Td} \cdot (8c^2/\sqrt{d} + 4 \sqrt{ (2m+1)(18 c^2 + R^2)}),
\end{align*}
where $c^2 = ( 32 \ln d - 8 ) \cdot  (8M_2^2 +  2\alpha^2 \Delta^2 (2m+1)(\frac{4}{\kappa})^{\frac2\kappa}  ).$
\end{proof}

\section{Restarted algorithms for strongly convex functions}\label{par:Restarts}

The restart technique is to run in cycle algorithm $\mathcal{A}$, taking the output point from the previous run as the initial point for the current one. 

\begin{algorithm}[ht!]
\caption{\algname{Restarted ZO-clipped-$\mathcal{A}$}}
\label{alg:R-clipped-SSTM}   
\begin{algorithmic}[1]
\REQUIRE Starting point $x^0$, number of restarts $N_r$,  number of iterations  $\{K_t\}_{t=1}^{N_r}$, algorithm $\mathcal{A}$, parameters $\{P_t\}_{t=1}^{N_r}$.
\STATE $\hat{x}^0 = x^0$.
\FOR{$t=1,\ldots, N_r$}
\STATE Run algorithms $\mathcal{A}$ with parameters $P_t$ and starting point $\hat{x}^{t-1}$. Set output point as $\hat{x}^{t}$.
\ENDFOR
\ENSURE $\hat x^{N_r}$
\end{algorithmic}
\end{algorithm}

Strong convexity of function $f$ with minimum $x^*$ implies an upper bound for the distance between point $x^K$ and solution $x^*$ as 
$$\frac{\mu}{2}\| x^K - x^*\|_2^2 \leq f(x) - f(x^*).$$
Considering the upper bounds on $f(x^K) - f(x^*)$ for our methods from Theorems~\ref{thm:SSTM col}, \ref{thm:SMD col}, one can construct a relation between $\|x_0 - x^*\|_2$ and $\|x^K - x^*\|_2$ after $K$ iterations. Based on this relation, one can calculate iteration, after which it is more efficient to start a new run rather than continue the current one  with slow convergence rate.

We apply the general Convergence Theorem~$2$ from \citep{kornilov2023accelerated} for \algname{R-ZO-clipped-SSTM} and Theorem~$5.2$ from \citep{kornilov2023gradient} for \algname{R-ZO-clipped-SMD} with oracle satisfying Assumption~\ref{as:oracle}. However, oracle can not depend on, $\tau$ which means that we should use  either Lipschitz oracle or one-point oracle with small noise, i.e., 
\begin{equation}\label{eq:ap_noise_upper_SSTM}
    \Delta  \leq \left(\frac{\kappa}{4} \right)^\frac1\kappa\frac{\varepsilon}{\sqrt{d}}.
\end{equation}
In Convergence Theorems, the minimal necessary value of $\tau = \frac{\varepsilon}{4M_2}$, hence 
\begin{eqnarray}
    \sigma^2 &=& 8d  M_2^2 +  2\left(\frac{d \Delta}{\tau}\right)^2 (2m+1)\left(\frac{4}{\kappa} \right)^{\frac2\kappa} \leq 32 (2m+1) \cdot d M_2^2 \notag.
\end{eqnarray}
\begin{theorem}[Convergence of Restarted \algname{ZO-clipped-med-SSTM}] \label{thm:R-med-SSTM cor}
     We denote $R_0 = \|x^0 - x^*\|_2$, where $x^0$ is a starting point. Consider $\mu$-strongly convex (As.~\ref{as:f convex}) and $M_2$-Lipschitz (As.~\ref{as:Lipshcitz}) function $f$ on $B_{3R_0 + 2\tau_1}(x^*)$ with oracle corrupted by noise under As.~\ref{as:dist} with $\Delta, \kappa > 0$. 
     
     Let $\varepsilon$ be desired accuracy, value  $1-\beta$ be desired probability and $N_r = \lceil \log_2(\nicefrac{\mu R_0^2}{2\varepsilon}) \rceil$ be the number of restarts. For each stage $t=1,...,N_r$, we run \algname{ZO-clipped-med-SSTM}  with  batch size $b_t$, median size $m_t = 2/\kappa + 1$, $\tau_t = \nicefrac{\varepsilon_t}{4M_2}, L_t = \nicefrac{M_2\sqrt{d}}{\tau_t}, K_t = \widetilde\Theta(\max\{\sqrt{\nicefrac{L_tR_{t-1}^2}{\varepsilon_t}}, (\nicefrac{\sigma R_{t-1}}{\varepsilon_t})^{2}/b_t\})$, $a_t = \widetilde\Theta(\max\{1, \nicefrac{\sigma K_t^{\frac32}}{\sqrt{b_t} L_tR_t}\})$ and $\lambda_k^t = \widetilde{\Theta}(\nicefrac{R_0}{\alpha_{k+1}^t})$, where $R_{t-1} = 2^{-\frac{(t-1)}{2}} R_0$, $\varepsilon_t = \nicefrac{\mu R_{t-1}^2}{4}$, $\ln \nicefrac{4N_r K_t}{\beta} \geq 1$, $\beta \in (0,1]$. 
 Then, to guarantee $f(\hat x^{N_r}) - f(x^*) \leq \varepsilon$ with probability at least $ 1 - \beta$, \algname{R-ZO-clipped-med-SSTM} requires

    \begin{itemize}

    \item \textbf{independent oracle under \eqref{eq:ap_noise_upper_SSTM}:}

\begin{eqnarray}
    \widetilde\cO\left((2m + 1) \cdot \max\left\{\sqrt{\frac{M_2^2\sqrt{d}}{\mu\varepsilon}}, \frac{dM_2^2}{\kappa \mu\varepsilon}\right\}\right) \text{ oracle calls,}\label{eq:R_med_STTM-total_one}
\end{eqnarray}

        \item \textbf{Lipschitz oracle:}

        \begin{eqnarray}
    \widetilde\cO\left((2m + 1)  \cdot \max\left\{\sqrt{\frac{M_2^2\sqrt{d}}{\mu\varepsilon}}, \frac{d(M_2^2 + d \Delta^2/\kappa^{\frac{2}{\kappa}})}{\mu\varepsilon}\right\}\right) \text{ 
 oracle calls.}\label{eq:R_med_STTM-total}
\end{eqnarray}
    \end{itemize}    
\end{theorem}
Similar to the convex case,  the first term in bounds \eqref{eq:R_med_STTM-total_one}, \eqref{eq:R_med_STTM-total} matches the optimal in $\varepsilon$ bound for the deterministic case for non-smooth strongly convex problems   (see \citep{bubeck2019complexity}). The second term matches the optimal in terms of $\varepsilon$  bound for zeroth-order problems with finite variance (see \citep{nemirovskij1983problem}).

Note that the function $f$ cannot be both strongly convex and Lipschitz on the entire space $\R^D$ \citep{grimmer2019convergence}. However, our original and restarted \algname{ZO-clipped-med-SSTM} need these assumptions hold true only on the initial ball $B_{3R_0 + 2\tau_1}(x^*)$, as they never leave it with high probability.

 \begin{theorem}[Convergence of Restarted \algname{ZO-clipped-med-SMD}]
     \label{thm:R-med-SMD cor}
    Consider $\mu$-strongly convex (As. \ref{as:f convex}) and $M_2$-Lipschitz (As. \ref{as:Lipshcitz}) function $f$ on $Q + B_{2\tau_1}(0)$ with two-point oracle corrupted by noise under As. \ref{as:dist} with $\kappa > 0$ and $\Delta > 0$. We set the prox-function $\Psi_p$ and norm $p \in [1,2]$, denote $R_0^2  \eqdef \sup_{x,y \in Q} 2 V_{\Psi_{p}}(x,y) $ for the diameter of the set $Q$ and  $R_t = R_0/2^t$.

Let $\varepsilon$ be desired accuracy and $N = \widetilde{O}\left(\frac{1}{2}\log_2\left(\frac{\mu R_0^2}{2\varepsilon}\right)\right)$ be the number of restarts. For each $t = \overline{1, N_r}$, we run \algname{ZO-clipped-med-SMD}  with  $K_t = \widetilde{O}\left( \left[ \frac{a_q\sigma}{\mu R_t} \right]^2 \right)$, $\tau_t = \frac{a_q\sigma R_t }{M_2 \sqrt{K_t}}$, $\lambda_t =  \sqrt{K_t}a_q\sigma$ and  $\nu_t = \frac{R_t}{\lambda_t}$. To guarantee $f(\hat x^{N_r}) - f(x^*) \leq \varepsilon$ with probability at least $ 1 - \beta$, \algname{R-ZO-clipped-med-SMD} requires

    \begin{itemize}

    \item \textbf{independent oracle under \eqref{eq:ap_noise_upper_SSTM}:}
        \begin{eqnarray}
    \widetilde\cO\left( (2m+1) \cdot  \frac{dM_2^2a^2_q}{\kappa \mu\varepsilon}\right) \text{ oracle calls}, \label{eq:R_med_SMD-total_one}
\end{eqnarray}
        \item \textbf{Lipschitz oracle:}
\begin{eqnarray}
    \widetilde\cO\left( (2m+1) \cdot  \frac{d(M_2^2 + d \Delta^2/\kappa^{\frac{2}{\kappa}}) a^2_q}{\mu\varepsilon}\right) \text{ oracle calls,} \label{eq:R_med_SMD-total}
\end{eqnarray}
        
    \end{itemize}
where $a_q = d^{\frac1q - \frac12}  \min \{ \sqrt{32\ln d - 8} , \sqrt{2q - 1}\}$.

\end{theorem}

\section{Experiments details} \label{sec:app exp}

Each experiment is computed on a CPU in several hours. The code is written in Python and will be made public after acceptance. For \algname{HTINF} \citep{huang2022adaptive}, \algname{APE} \citep{lee2020optimal}, \algname{ZO-clipped-SSTM} and  \algname{ZO-clipped-SGD} \citep{kornilov2023accelerated},   we provide our own implementation based on pseudocodes from the original articles.

\subsection{Multi-armed bandits}\label{sec:bandit extra exp}
In our experimental setup, individual experiments are subject to significant random deviations. To enhance the informativeness of the results, we conduct $100$ individual experiments and analyze aggregated statistics.

By design, we possess knowledge of the conditional probability of selecting the optimal arm for all algorithms, which remains stochastic due to the nature of the experiment's history.

To mitigate the high dispersion in probabilities, we apply an average filter with a window size of $30$ to reduce noise in the plot. \algname{APE} and \algname{HTINF} can't handle cases when noise expectation is unbounded, so we modeled this case with a low value of $\alpha=0.01$, where $1 + \alpha$ is the moment that exists in the problem statement for \algname{APE} and \algname{HTINF}.

\subsubsection{Dependence on $\kappa$}
We conduct experiments to check dependence on $\kappa$ under the symmetric Levy $\alpha$-stable noise, where $\alpha = \kappa$. We compare standard \algname{INFC} method from \citep{dorn2023} which allows $\kappa \leq 1$ with our \algname{Clipped-INF-med-SMD}, and comparison results can be found in Figure \ref{fig:kappa_bandits}. 

\begin{figure*}[!h]
    \centering
    \includegraphics[width=0.9\textwidth]{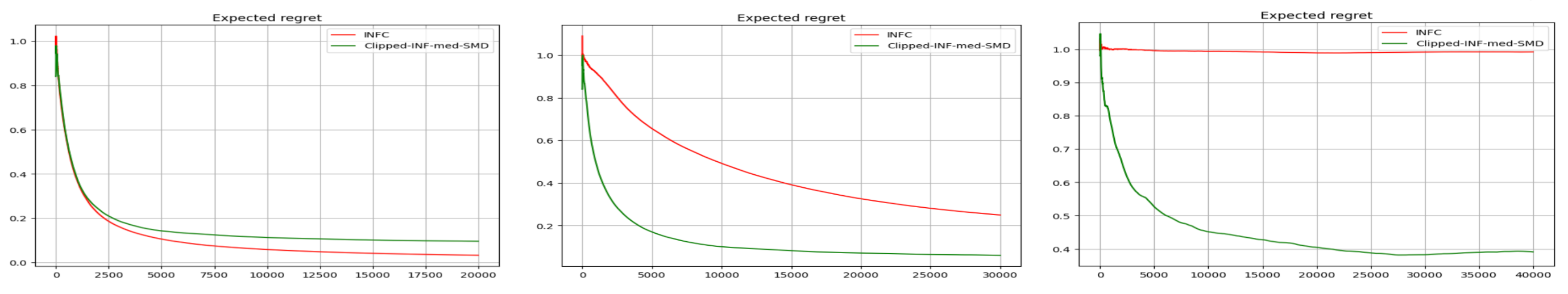}
    \caption{Convergence of our \algname{Clipped-INF-med-SMD} and \algname{INFC} under $\kappa = 1.5, 1, 0.5$. }
    \label{fig:kappa_bandits}
\end{figure*}

\subsection{Zeroth-order optimization}\label{sec:ZO extra exp}

To generate $A \in \mathbb{R}^{l\times d}$ and $b \in \mathbb{R}^{l}$ we draw them from standard normal distribution with $d = 16$ and $l = 200$. For algorithms, we gridsearch stepsize $a$ over $\{0.1,0.01,0.001,0.0001\}$ and smoothing parameter $\tau$ over  $\{0.1,0.01,0.001\}$. For \algname{ZO-clipped-med-SSTM}, the parameters $a = 0.001$, $L = 1$ (note that $a$ and $L$ are actually used together in the algorithm, therefore, we gridsearch only one of them) and $\tau = 0.01$ are the best. For \algname{ZO-clipped-med-SGD}, we use $a = 0.01$, default momentum of $0.9$ and $\tau = 0.1$. For non-median versions, after the same gridsearch, parameters happened to be the same.

\end{document}